\documentclass[11pt]{amsart}  
\usepackage{amsmath}
\usepackage{amsfonts}
\usepackage{amssymb}
\usepackage{latexsym}
\usepackage{graphicx}
\usepackage{subfigure}
\usepackage{color}
\usepackage{mathtools}
\usepackage{hyperref}
\usepackage{verbatim}
\usepackage[all]{xy}
\usepackage{graphics}
\usepackage{pdfsync}
\usepackage{youngtab}
\usepackage{ytableau}
\usepackage{young}
\usepackage{cancel} % for gigantic cancellation marks
\usepackage[normalem]{ulem} % for strike throughs
\usepackage{calrsfs}
\DeclareMathAlphabet{\pazocal}{OMS}{zplm}{m}{n}
\usepackage{caption}

\usepackage{tikz}
\usetikzlibrary{arrows,petri,topaths}
\usepackage{tkz-berge}
\usepackage{xcolor}
\usepackage{tikz-qtree}
\usetikzlibrary{positioning,arrows,calc,intersections,shapes,decorations}

\theoremstyle{plain}
\newtheorem{theorem}{Theorem}[section]
\newtheorem{proposition}[theorem]{Proposition}

\newtheorem{lemma}[theorem]{Lemma}
\newtheorem{corollary}[theorem]{Corollary}

\theoremstyle{remark}
\newtheorem{remark}[theorem]{Remark}

\numberwithin{equation}{section}

\newcommand{\tcr}[1]{\textcolor{red}{#1}}
\newcommand{\tcb}[1]{\textcolor{blue}{#1}}

\newcommand{\des}{{\sf des}} % for descent sets
\newcommand{\suchthat}{\;|\;}

\definecolor{darkbrown}{rgb}{0.6, 0.3, 0.1}
\newcommand{\x}{\mathbf{x}}

\newcommand{\p}[1]{\pc{#1}}
\newcommand{\flatten}[1]{\mathsf{flat}(#1)}
\newcommand{\precdot}{\prec\mathrel{\mkern-5mu}\mathrel{\cdot}}
\newcommand{\ombfo}{\omega_{\mathbf{o}}}
\newcommand{\bfo}{\mathbf{o}}
\newcommand{\slide}[1]{\mathfrak{F}_{#1}} % slides polynomials
\newcommand{\mc}[1]{\mathcal{#1}}
\newcommand{\pc}[1]{\pazocal{#1}}
\newcommand{\vbar}{\hspace{0.2mm}|\hspace{0.2mm}}
\newcommand{\backslide}[1]{\overleftarrow{\mathfrak{F}}_{#1}} % the stable version of slides

\begin{document}

\title[Slide positivity of chromatic nonsymmetric polynomials]{Chromatic nonsymmetric polynomials of Dyck graphs are slide-positive}
\author{Vasu Tewari}
\address{Department of Mathematics, University of Pennsylvania, Philadelphia, PA 19104, USA}
\email{\href{mailto:vvtewari@math.upenn.edu}{vvtewari@math.upenn.edu}}
\author{Andrew Timothy Wilson}
\address{Department of Mathematics, Portland State University, Portland, OR 97201, USA}
\email{\href{mailto:atwilson0328@gmail.com}{atwilson0328@gmail.com}}
\author{Philip B. Zhang}
\address{College of Mathematical Science, Tianjin Normal University, Tianjin 300387, China}
\email{\href{mailto:zhang@tjnu.edu.cn}{zhang@tjnu.edu.cn}}

\subjclass[2010]{Primary 05E05; Secondary 05A05, 05C15}
\keywords{}

\begin{abstract}
Motivated by the study of Macdonald polynomials, J.~Haglund and A.~Wilson introduced a nonsymmetric polynomial analogue of the chromatic quasisymmetric function called the \emph{chromatic nonsymmetric polynomial} of a Dyck graph.
We give a positive expansion for this polynomial in the basis of fundamental slide polynomials using recent work of Assaf-Bergeron on flagged $(P,\rho)$-partitions.
We then derive the known expansion for the chromatic quasisymmetric function of Dyck graphs in terms of Gessel's fundamental basis by taking a backstable limit of our expansion.
\end{abstract}

\maketitle
%\tableofcontents
\section{Introduction}\label{sec:introduction}
The chromatic polynomial was introduced by Birkhoff \cite{Birkhoff} in 1912 for planar graphs in an attempt to establish the four color theorem, and was generalized to arbitrary graphs by Whitney \cite{Whitney}.
This polynomial and its various generalizations have proven to be fertile grounds for a host of interesting mathematics ever since.
Stanley \cite{Stanley-chromatic} introduced a symmetric function generalization of the chromatic polynomial called the chromatic symmetric function.
This function was studied from the perspective of $(P,\omega)$-partitions and quasisymmetric functions by Chow \cite{Chow}.
By introducing another parameter $t$, this perspective gives a more refined function called the chromatic quasisymmetric function, which was introduced by Shareshian-Wachs \cite{ShWa}.
They established that the chromatic quasisymmetric functions of incomparability graphs of posets  expand in terms of Gessel's fundamental quasisymmetric functions with coefficients in $\mathbb{N}[t]$.
For incomporability graphs of natural unit interval orders, the chromatic quasisymmetric function is in fact Schur positive. 
See \cite{Alexandersson-Panova, Athanasiadis, Brosnan-Chow,MGP} for some remarkable aspects of these symmetric functions and related topics.
Given the recent interest in polynomial analogues of combinatorially-defined quasisymmetric and symmetric functions
\cite{Pechenik-Searles}, it is natural to investigate chromatic quasisymmetric functions.
For the class of graphs mentioned earlier, a nonsymmetric polynomial analogue was proposed by Haglund-Wilson \cite{Haglund-Wilson}, and this is the chief object of our study.

These graphs are conveniently encoded via Dyck paths, and Novelli-Thibon \cite{Novelli-Thibon}, in their study of the attached chromatic quasisymmetric functions from the viewpoint of Hopf algebras, refer to them as \emph{Dyck graphs}. To allow for a nonsymmetric polynomial analogue, Haglund-Wilson \cite{Haglund-Wilson} attached Dyck graphs to partial Dyck paths and used squares that lie between $D$ and the line $y=x$ to impose restrictions on the colors allowed at each vertex.
Thus, instead of having a common set of colors for all vertices, we use $D$ to obtain different restrictions on colors for different vertices.
Taking the generating function of monomials attached to proper colorings with these restrictions  along with appropriate $t$-weights gives us the chromatic nonsymmetric polynomial $\p X_{D}(\x_r;t)$.

Our central result is that $\p X_D(\x_r;t)$ expands in terms of fundamental slide polynomials with coefficients in $\mathbb{N}[t]$.
The fundamental slides \cite{Assaf-Searles}  are polynomial analogues of fundamental quasisymmetric functions.
By extending our partial Dyck path to an infinite path by prepending infinitely many east steps, we can obtain a formal power series $\overleftarrow{\p X}_D(\x;t)$ as the stable limit.
We then obtain as a corollary of our main result the expansion of the chromatic quasisymmetric function of Dyck graphs in terms of  fundamental quasisymmetric functions by way of $\overleftarrow{\p X}_D(\x;t)$.
	
\smallskip

\textbf{Outline of the article: }
After setting up the necessary background,  we introduce chromatic nonsymmetric polynomials at the end of Section~\ref{subsec:dyck_paths_and_graphs}.
In Sections~\ref{subsec:p_partitions} and \ref{subsec:lin_orders} we define restricted $(P,\omega)$-partitions following Assaf-Bergeron \cite{Assaf-Bergeron} and introduce the polynomial analogue of the quasisymmetric function attached to usual $(P,\omega)$-partitions, focusing in particular on labeled linear orders.
In Section~\ref{sec:main_result}, we provide a positive expansion for the chromatic nonsymmetric polynomial in terms of fundamental slide polynomials in Theorem~\ref{thm:main_theorem}.
We then proceed to study the backstable limit, drawing inspiration from work of Lam-Lee-Shimozono \cite{Lam-Lee-Shimozono}, and derive a known expansion of the chromatic quasisymmetric function for Dyck graphs  in terms fundamental quasisymmetric functions as a corollary of our main result.

%%%%%%%%%%%%%%%%%%%%%%%%%%%%%%%%%%%%
\section{Background}\label{sec:background}
%%%%%%%%%%%%%%%%%%%%%%%%%%%%%%%%%%%%

For $n$ a nonnegative integer, we denote  the set $\{1,\dots, n\}$ by $[n]$. Throughout, we use $<$ to denote the natural order on integers.
Given a positive integer $n$, we denote by $\x_n$ the commutative alphabet $\{x_1,\dots,x_n\}$.
For notions concerning symmetric/quasisymmetric functions and standard combinatorial constructions that are not defined here, we refer the reader to \cite{Macdonald,Sagan,Stanley-ec2}.

%%%%%%%%%%%%%%%%%%%%%%%%%%%
\subsection{Graphs, colorings}\label{subsec:colorings_etc}
%%%%%%%%%%%%%%%%%%%%%%%%%%

We consider finite simple graphs $G=(V,E)$ where $V$ is an ordered set of vertices. 
We identify $V$ with $[n]$ with the order being the natural order on the integers.
The set of edges $E$  is a subset of $\{\{i,j\}\suchthat 1\leq i<j\leq n\}$.
%Whenever we refer to an edge $(i,j)$ in $G$, we assume that $i<j$.
A \emph{coloring} of $G$ is a map $f : [n] \to\mathbb{Z}$. 
For the most part we will restrict to maps to $\mathbb{Z}_{>0}$, though we allow for `negative' colors  in Subsection~\ref{subsec:backstable}.
A coloring $f$ of $G$ is \emph{proper}  if $f(i)\neq f(j)$ for all edges $\{i,j\}\in E$. 
We call $\{i,j\}\in E$ a \emph{descent} of $f$ if $i<j$ and $f(i)>f(j)$.
We denote the number of descents in $f$ by $\des_G(f)$.
The \emph{chromatic quasisymmetric function} introduced by Shareshian and Wachs \cite[Definition 1.2]{ShWa} is defined as follows:
\begin{align}\label{eqn:chromatic_quasisymmetric}
X_{G}(\x_{\geq 1};t)=\sum_{\text{proper colorings }f:[n]\to \mathbb{Z}_{> 0}}t^{\des_G(f)}x_{f(1)}\cdots x_{f(n)},
\end{align}
where $n$ is the cardinality of $V$ and $\x_{\geq 1}$ denotes the alphabet of commuting indeterminates $\{x_i\suchthat i\in \mathbb{Z}_{>0}\}$.
We remark here that the definition in ~\eqref{eqn:chromatic_quasisymmetric} differs from that in \cite{ShWa} 
up to twisting by an involution defined on the ring of quasisymmetric functions. Since this is a minor point, we abuse notation and refer to the function in ~\eqref{eqn:chromatic_quasisymmetric} as the chromatic quasisymmetric function of $G$.

%%%%%%%%%%%%%%%%%%%%%%%%
\subsection{Partial Dyck paths and associated graphs}\label{subsec:dyck_paths_and_graphs}
%%%%%%%%%%%%%%%%%%%%%%%%

Let $n$ and $r$ be nonnegative integers. 
We define $P_{n,r}$ to be set of lattice paths that begin at $(0,r)$, end at $(n+r,n+r)$, take unit north and east steps, and stay weakly above the line $y=x$.
We refer to elements of $P_{n,r}$ as \emph{partial Dyck paths}.
We next discuss  a procedure that assigns to each $D\in P_{n,r}$ a graph $G_D=([n],E_D)$ and a function $\rho_D:[n]\to \mathbb{Z}_{\geq 0}$.

Given $D\in P_{n,r}$, assign the integers $1,\dots,n+r$ to the unit squares along the diagonal $y=x$ going from $(0,0)$ to $(n+r,n+r)$, and the integers $0,-1,\dots$ in the opposite direction.
For $p < q$, let ${\sf s}(p,q)$ refer the unique square in the plane directly north of the square labeled $p$ and directly west of the square labeled $q$.
We define $G_D$ by explicitly describing $E_D$ \textemdash{} $\{i,j\}\in E_D$ if and only if $i<j$ and ${\sf s}(i+r,j+r)$ lies below $D$.
Following \cite{Novelli-Thibon}, we call $G_D$ a \emph{Dyck graph}.
Such graphs are characterized by the property that $\{i,j\}\in E$ and $i<j$ implies $\{i',j'\}\in E$  for all $i\leq i'<j'\leq j$.
The \emph{restriction map} $\rho_D$ is defined as follows: for every $1\leq i\leq n$, find the largest $j\leq r$ such that ${\sf s}(j,i+r)$ lies above $D$ and set $\rho_D(i)=j$.
Note in particular that $0\leq \rho_D(1)\leq \rho_D(2)\leq \cdots \leq \rho_D(n)\leq r$. 

Figure~\ref{fig:partial_dyck} shows a partial Dyck path $D$ in $P_{6,5}$.
Thus the vertex set of $G_D$ is $[6]$. 
As the green square ${\sf s}(6,8)$ lies below $D$, we infer that $(1,3)\in E_D$. Arguing in this manner, one can compute $E_D$.
As the square ${\sf s}(4,7)$ does not lie below $D$ while ${\sf s}(5,7)$ does, we infer that $\rho_D(2)=4$.
Figure~\ref{fig:g_d} shows $G_D$ as well as the restriction map $\rho_D$. The latter is written with inequalities the meaning of which we now clarify.

Given $D\in P_{n,r}$, Haglund-Wilson \cite[Section 5.6]{Haglund-Wilson} introduce a polynomial $\p X_D(\x_r;t)$ by mimicking the definition of the chromatic quasisymmetric function.
We have
\begin{align}\label{eqn:chromatic_nonsymmetric_polynomial}
\p X_{D}(\x_r;t)=\sum_{\substack{\text{proper colorings }f:[n]\to \mathbb{Z}_{> 0}\\f(i)\leq \rho_D(i)}}t^{\des_G(f)}x_{f(1)}\cdots x_{f(n)}.
\end{align}
The polynomial $\p X_{D}(\x_r;t)$ is the \emph{chromatic nonsymmetric polynomial} of $G_D$.
See Section~\ref{sec:final} for more context on its definition.

\medskip

\begin{minipage}{\linewidth}
\begin{minipage}{0.5\linewidth}
\centering
\begin{tikzpicture}[scale=.4]
    \draw[gray,very thin] (0,0) grid (13,13);
    \draw[line width=0.25mm, black, <->] (0,1)--(13,1);
    \draw[line width=0.25mm, black, <->] (1,0)--(1,13);
    
    \draw[rounded corners,fill=green,opacity=0.3] (6, 8) rectangle (7,9) {}; 

    \draw[green, line width=0.3mm] (6.5,7)--(6.5,8.5)--(8,8.5);
    \draw[rounded corners,fill=blue,opacity=0.3] (0, 0) rectangle (1, 1) {}; 
    \draw[rounded corners,fill=blue,opacity=0.3] (1, 1) rectangle (2, 2) {}; 
    \draw[rounded corners,fill=blue,opacity=0.3] (2, 2) rectangle (3, 3) {}; 
    \draw[rounded corners,fill=blue,opacity=0.3] (3, 3) rectangle (4, 4) {}; 
    \draw[rounded corners,fill=blue,opacity=0.3] (4, 4) rectangle (5, 5) {}; 
    \draw[rounded corners,fill=blue,opacity=0.3] (5, 5) rectangle (6, 6) {}; 
    
    \draw[rounded corners,fill=red,opacity=0.3] (6, 6) rectangle (7, 7) {}; 
    \draw[rounded corners,fill=red,opacity=0.3] (7, 7) rectangle (8, 8) {}; 
    \draw[rounded corners,fill=red,opacity=0.3] (8, 8) rectangle (9, 9) {}; 
    \draw[rounded corners,fill=red,opacity=0.3] (9, 9) rectangle (10, 10) {}; 
    \draw[rounded corners,fill=red,opacity=0.3] (10, 10) rectangle (11, 11) {}; 
    \draw[rounded corners,fill=red,opacity=0.3] (11, 11) rectangle (12, 12) {}; 
    
    \draw[rounded corners,fill=orange,opacity=0.3] (4, 7) rectangle (5,8) {}; 
    \draw[orange, line width=0.3mm] (4.5,5)--(4.5,7.5)--(7,7.5);
       
    \node[draw, circle,minimum size=3pt,inner sep=0pt, outer sep=0pt, fill=blue] at (1, 6)   (start) {};
    \node[draw, circle,minimum size=3pt,inner sep=0pt, outer sep=0pt, fill=blue] at (12, 12)   (end) {}; 
    \node[] at (6.5, 6.5)   (v1) {$6$}; 
  \node[] at (7.5, 7.5)   (v2) {$7$}; 
  \node[] at (8.5, 8.5)   (v3) {$8$}; 
  \node[] at (9.5, 9.5)   (v4) {$9$}; 
  \node[] at (10.5, 10.5)   (v5) {$10$}; 
  \node[] at (11.5, 11.5)   (v6) {$11$}; 
     
   \node[] at (0.5, 0.5)   (c0) {$0$}; 
  \node[] at (1.5, 1.5)   (c1) {$1$}; 
  \node[] at (2.5, 2.5)   (c2) {$2$}; 
  \node[] at (3.5, 3.5)   (c3) {$3$}; 
  \node[] at (4.5, 4.5)   (c4) {$4$}; 
  \node[] at (5.5, 5.5)   (c5) {$5$}; 
  
    \draw[blue, line width=0.3mm] (1,6)--(2,6);
    \draw[blue, line width=0.3mm] (2,6)--(2,7);
    \draw[blue, line width=0.3mm] (2,7)--(3,7);
    \draw[blue, line width=0.3mm] (3,7)--(4,7);
    \draw[blue, line width=0.3mm] (4,7)--(5,7);
    \draw[blue, line width=0.3mm] (5,7)--(5,8);
    \draw[blue, line width=0.3mm] (5,8)--(6,8);
    \draw[blue, line width=0.3mm] (6,8)--(6,9);
    \draw[blue, line width=0.3mm] (6,9)--(7,9);
    \draw[blue, line width=0.3mm] (7,9)--(8,9);
    \draw[blue, line width=0.3mm] (8,9)--(8,10);
    \draw[blue, line width=0.3mm] (8,10)--(8,11);
    \draw[blue, line width=0.3mm] (8,11)--(9,11);
    \draw[blue, line width=0.3mm] (9,11)--(10,11);
    \draw[blue, line width=0.3mm] (10,11)--(10,12);
    \draw[blue, line width=0.3mm] (10,12)--(11,12);
    \draw[blue, line width=0.3mm] (11,12)--(12,12);
  \end{tikzpicture}
\captionof{figure}{A partial Dyck path $D$.}
\label{fig:partial_dyck}
\end{minipage}
\begin{minipage}{0.45\linewidth}
\centering
  \begin{tikzpicture}[scale=1]
    \node[draw, circle,minimum size=4pt,inner sep=0pt, outer sep=0pt, fill=black] at (0, 0) [label=right:1]   (v1) {}; 
    \node[draw, circle,minimum size=4pt,inner sep=0pt, outer sep=0pt, fill=black] at (1, 0) [label=right:2]  (v2) {}; 
    \node[draw, circle,minimum size=4pt,inner sep=0pt, outer sep=0pt, fill=black] at (2,0)  [label=right:3] (v3) {}; 
    \node[draw, circle,minimum size=4pt,inner sep=0pt, outer sep=0pt, fill=black] at (3, 0) [label=right:4]  (v4) {}; 
    \node[draw, circle,minimum size=4pt,inner sep=0pt, outer sep=0pt, fill=black] at (4, 0)[label=right:5]   (v5) {}; 
    \node[draw, circle,minimum size=4pt,inner sep=0pt, outer sep=0pt, fill=black] at (5, 0)  [label= right:6] (v6) {}; 
    
    \node[] at (0,-0.5)    (ineq1) {\textcolor{red}{\tiny{$\leq\! 1$}}}; 
    \node[] at (1, -0.5)   (ineq2) {\textcolor{red}{\tiny{$\leq\!4$}}}; 
    \node[] at (2,-0.5)  (ineq3) {\textcolor{red}{\tiny{$\leq\! 5$}}}; 
    \node[] at (3, -0.5)  (ineq4) {\textcolor{red}{\tiny{$\leq\! 5$}}}; 
    \node[] at (4, -0.5)  (ineq5) {\textcolor{red}{\tiny{$\leq\! 5$}}}; 
    \node[] at (5, -0.5)   (ineq6) {\textcolor{red}{\tiny{$\leq\! 5$}}}; 
    
    \draw [black,thick]   (v1) to[out=60,in=120] (v2);
    \draw [black,thick]   (v1) to[out=90,in=90] (v3);
    \draw [black,thick]   (v2) to[out=60,in=120] (v3);
    \draw [black,thick]   (v3) to[out=60,in=120] (v4);
    \draw [black,thick]   (v3) to[out=90,in=90] (v5);
    \draw [black,thick]   (v4) to[out=60,in=120] (v5);
    \draw [black,thick]   (v5) to[out=60,in=120] (v6);
    
  \end{tikzpicture}
\captionof{figure}{The graph $G_D$.}
\label{fig:g_d}
\end{minipage}
\end{minipage}

%%%%%%%%%%%%%%%%%%%%%%%%%
\subsection{Slide polynomials}\label{subsec:slides}
%%%%%%%%%%%%%%%%%%%%%%%%%
We recall some notions before defining slide polynomials, a polynomial analogue of the fundamental quasisymmetric functions introduced in \cite{Assaf-Searles}.
Our treatment is slightly nonstandard, but it will allow us to deal with stable limits in a uniform manner.

Given a sequence of nonnegative integers ${\bf a}=(a_i)_{i\in \mathbb{Z}}$, we define the \emph{support} of ${\bf a}$, denoted by $\mathrm{supp}({\bf a})$ to be the set $\{i\in \mathbb{Z}\suchthat a_i>0\}$.
If $\mathrm{supp}({\bf a})$ is finite, then we call ${\bf a}$  a \emph{weak composition}.
We denote the set of weak compositions by $S_{\mathbb{Z}}$ because we may interpret ${\bf a}$ as the code\footnote{Recall that the code ${\bf a}=(a_i)_{i\in \mathbb{Z}}$ of a permutation $w=\cdots w_{-1}w_0w_1\cdots$ of $\mathbb{Z}$ is defined by setting $a_i=|\{j>i\suchthat w_i>w_j \}|$.} of a permutation of $\mathbb{Z}$ that fixes all but finitely many integers. 
The \emph{weight} of any sequence (finite or infinite) is the sum of its entries.
There is a unique weak composition of weight $0$: the sequence consisting solely of $0$s.
Given an integer $r$, the set of all weak compositions ${\bf a}$ that satisfy $a_i=0$ for all $i>r$ is denoted by $S_{\mathbb{Z}}^r$.
The (potentially empty) sequence obtained by omitting all zeros from a weak composition is called a \emph{strong  composition}.
We denote the strong composition underlying ${\bf a}$ by $\flatten{{\bf a}}$. The unique strong composition of weight 0 is denoted by $\varnothing$.
\textbf{From this point onward, we reserve the term composition for weak compositions}.

For two strong compositions $\alpha$ and $\beta$ of the same weight, we say that  $\alpha$ refines $\beta$ if we can 
iteratively combine adjacent parts of $\alpha$ to obtain $\beta$.
For instance, $(2,1,1,2,1)$ refines $(3,4)$. 
Using the notion of refinement we define a total order $\leq_{\mathrm{c}}$ on  compositions of the same weight as follows: ${\bf b}\leq_{\mathrm{c}} {\bf a}$ if $\flatten{{\bf b}}$ refines $\flatten{{\bf a}}$ and additionally, ${\bf b}$ is 
smaller than ${\bf a}$ in reverse lexicographic order.
We denote the set of all ${\bf b} \leq_{\mathrm{c}} {\bf a}$ by $\p C_{\leq {\bf a}}$.
Given a positive integer $r$, a distinguished subset of $\p C_{\leq {\bf a}}$, denoted by $\p C_{\leq {\bf a}}^{(r)}$, comprises those ${\bf b}$ that satisfy $\mathrm{supp}({\bf b})\subseteq [r]$.
Note that in contrast to  $\p C_{\leq {\bf a}}$ which is infinite except when ${\bf a}$ has weight zero, the set $\p C_{\leq {\bf a}}^{(r)}$ is finite.

\begin{remark}
For the sake of clarity, when dealing with explicit instances of compositions ${\bf a}=(\dots,a_{-1},a_0,a_1,\dots)$ we suppress leading and trailing zeros and furthermore place a bar between $a_0$ and $a_1$ to clearly show where the positively indexed terms in the sequence begin.
If $a_i=0$ for all $i\leq 0$, we omit the bar and write ${\bf a}$ as a finite sequence, which is the more conventional form.
\end{remark}

 Let ${\bf a}=(0,2,0,2)\in S_{\mathbb{Z}}^4$. Then $\flatten{{\bf a}}=(2,2)$.
In ~\eqref{eqn:all_elements}, we list all elements of $\p C_{\leq {\bf a}}^{(4)}$ after omitting brackets and commas for brevity.
\begin{align}\label{eqn:all_elements}
\{0202, 2002, 2020,2200,1102,1120,1111,0211,2011,2101,2110\}.
\end{align}
On the other hand, if ${\bf a}=(1\vbar 0,2,0,1)$, then $\p C_{\leq {\bf a}}^{(4)}$ is clearly empty as compositions ${\bf b}\leq_{\mathrm{c}}  (1\vbar 0,2,0,1)$ cannot satisfy $\mathrm{supp}({\bf b})\subseteq [4]$.

Given a positive integer $r$ and a composition ${\bf a}\in S_{\mathbb{Z}}$, the \emph{fundamental slide polynomial} $\slide{{\bf a}}(\x_r)$ \cite[Definition 3.6]{Assaf-Searles} is defined as
\begin{align}\label{eqn:slides}
\slide{{\bf a}}(\x_r)\coloneqq\sum_{{\bf b}\in \p C_{\leq {\bf a}}^{(r)}}x_1^{b_1}\cdots x_{r}^{b_r}.
\end{align}
Henceforth we refer to fundamental slide polynomials as \emph{slide polynomials}.
The expansion of the slide polynomial indexed by ${\bf a} = (0,2,0,1)$ is 
\begin{align}\label{eqn:slide_ex}
\slide{0201}(\x_4)=x_2^2x_4+x_1^2x_4+x_1^2x_3+x_1^2x_2+x_1x_2x_4+x_1x_2x_3.
\end{align}
If ${\bf a}=(1\vbar 0,2,1,0)$, then $\slide{{\bf a}}(\x_4)=0$.

A simple triangularity argument \cite[Theorem 3.9]{Assaf-Searles} implies that the set of slide polynomials $\slide{{\bf a}}(\x_r)$ as ${\bf a}$ ranges over compositions satisfying $\mathrm{supp}({\bf a})\subseteq [r]$ is a basis for the polynomial ring $\mathbb{Q}[x_1,\dots,x_r]$.
We refer the reader to \cite{Assaf-Searles} for other aspects of slide polynomials, in particular the relation to Schubert polynomials.

%%%%%%%%%%%%%%%%%%%%%%%%%
\subsection{Restricted $(P,\omega)$-partitions}\label{subsec:p_partitions}
%%%%%%%%%%%%%%%%%%%%%%%%%%

All our posets  are finite. 
Given a poset $P$, we always assume that its ground set is identified with $[|P|]$.
We depict $P$ using its Hasse diagram. 
We use $\preceq_P$ to denote the order relation on $P$, and cover relations are denoted by $\precdot_P$.
A \emph{labeling} of $P$ is a map $\omega:[n]\to [n]$ where $n\coloneqq |P|$.
We refer to the ordered pair $(P,\omega)$ as a \emph{labeled poset}.
A $(P,\omega)$-partition is a map $f:[n]\to \mathbb{Z}_{>0}$ satisfying the conditions that
\begin{enumerate}
\item if $i\precdot_P j$ and $\omega(i)<\omega(j)$, then $f(i)\leq f(j)$.
\item if $i\precdot_P j$ and $\omega(i)>\omega(j)$, then $f(i)< f(j)$.
\end{enumerate}
Given a map $\rho:[n]\to \mathbb{Z}$, we refer to the datum $(P,\omega,\rho)$ as a \emph{$\rho$-restricted labeled poset}. 
Additionally, a $(P,\omega)$-partition $f$ that satisfies $f(i)\leq \rho(i)$ for all $i\in [n]$ is said to be a \emph{$\rho$-restricted $(P,\omega)$-partition} \cite{Assaf-Bergeron}.
We denote the set of $\rho$-restricted $(P,\omega)$-partitions by $\mathcal{A}(P,\omega,\rho)$.

\begin{remark}
Assaf-Bergeron \cite{Assaf-Bergeron} work under the assumption that the ground set of $P$ is identified with $[|P|]$ via the labeling $\omega$.
Since we will 
talk about graphs and posets arising from acyclic orientations thereon, the ground set of our posets will be vertex set of our graph (already identified with $[n]$ for some $n$); the labelings $\omega$ we employ might be different. % than this inherited labeling.
\textbf{Throughout this article, in our Hasse diagrams, the numbers within nodes correspond to the labeling $\omega$}.
\end{remark}

Assaf-Bergeron \cite[Section 3]{Assaf-Bergeron} associate a polynomial $\p {F}_{(P,\omega,\rho)}$ with the triple $(P,\omega,\rho)$, mimicking the classical theory of quasisymmetric functions attached to $(P,\omega)$-partitions.
More specifically, define 
$\p {F}_{(P,\omega,\rho)}$ as
\begin{align}\label{eqn:p_partition_gf}
\p {F}_{(P,\omega,\rho)}\coloneqq\sum_{f\in \mathcal{A}(P,\omega,\rho)} x_{f(1)}\cdots x_{f(n)}.
\end{align}
It is possible that $\mathcal{A}(P,\omega,\rho)$ is empty, for instance if $\rho$ takes a value in $\mathbb{Z}_{\leq 0}$. In such cases $\p {F}_{(P,\omega,\rho)}$  equals $0$.
For the triple $(P,\omega,\rho)$ on the left in Figure~\ref{fig:ex_computation}, assume that the ground set of $P$ is identified with $[3]$ via $\omega$. 
One can check that
\begin{align}
\p {F}_{(P,\omega,\rho)}&=(\tcr{x_1x_2^2+x_1x_2x_3})+(\tcb{x_2^2x_3+x_1x_2x_3+x_1^2x_3+x_1^2x_2})\nonumber\\&=
\tcr{\slide{120}(\x_3)+\slide{111}(\x_3)}+\tcb{\slide{021}(\x_3)}.
\end{align}
Observe that summands within the first pair of parentheses are from the linear order in the middle in Figure~\ref{fig:ex_computation}, whereas those within the second pair  are from the linear order on the right.

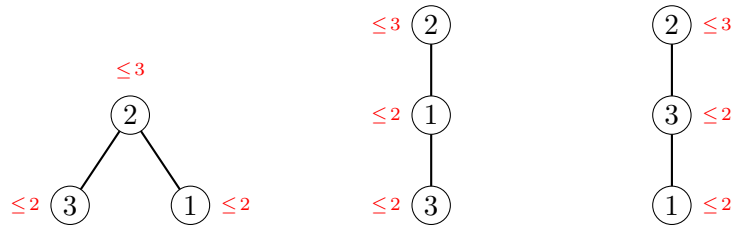
\begin{figure}[h]
  \begin{tikzpicture}[scale=0.8]    
    \node[draw, circle,minimum size=10pt,inner sep=2pt, outer sep=0pt] at (-2, 0) (v1) {$3$}; 
    \node[draw, circle,minimum size=10pt,inner sep=2pt, outer sep=0pt] at (-1, 1.5) (v2) {$2$}; 
    \node[draw, circle,minimum size=10pt,inner sep=2pt, outer sep=0pt] at (0,0)   (v3) {$1$};   
    \node[] at (-2.75,0)    (ineq1) {\textcolor{red}{\tiny{$\leq\! 2$}}}; 
    \node[] at (-1, 2.25)   (ineq2) {\textcolor{red}{\tiny{$\leq\!3$}}}; 
    \node[] at (0.75,0)  (ineq3) {\textcolor{red}{\tiny{$\leq\! 2$}}};   
    \draw [black,thick]   (v1) to (v2);
    \draw [black,thick]   (v2) to (v3);
    
     \node[draw, circle,minimum size=10pt,inner sep=2pt, outer sep=0pt] at (4, 0) (u1) {$3$}; 
    \node[draw, circle,minimum size=10pt,inner sep=2pt, outer sep=0pt] at (4, 1.5) (u2) {$1$}; 
    \node[draw, circle,minimum size=10pt,inner sep=2pt, outer sep=0pt] at (4,3)   (u3) {$2$};   

 \node[draw, circle,minimum size=10pt,inner sep=2pt, outer sep=0pt] at (8, 0) (w1) {$1$}; 
    \node[draw, circle,minimum size=10pt,inner sep=2pt, outer sep=0pt] at (8, 1.5) (w2) {$3$}; 
    \node[draw, circle,minimum size=10pt,inner sep=2pt, outer sep=0pt] at (8,3)   (w3) {$2$};   
   
    \node[] at (3.25,0)    (ineq1) {\textcolor{red}{\tiny{$\leq\! 2$}}}; 
    \node[] at (3.25, 1.5)   (ineq2) {\textcolor{red}{\tiny{$\leq\!2$}}}; 
    \node[] at (3.25,3)  (ineq3) {\textcolor{red}{\tiny{$\leq\! 3$}}};   

\node[] at (8.75,0)    (ineq4) {\textcolor{red}{\tiny{$\leq\! 2$}}}; 
    \node[] at (8.75, 1.5)   (ineq5) {\textcolor{red}{\tiny{$\leq\!2$}}}; 
    \node[] at (8.75,3)  (ineq6) {\textcolor{red}{\tiny{$\leq\! 3$}}}; 

    \draw [black,thick]   (u1) to (u2);
    \draw [black,thick]   (u2) to (u3);

    \draw [black,thick]   (w1) to (w2);
    \draw [black,thick]   (w2) to (w3);

  \end{tikzpicture}
  \caption{A $\rho$-restricted labeled poset and its two linear extensions.}
  \label{fig:ex_computation}
\end{figure}
 
It is worth remarking that even though in our earlier example
$\p F_{(P,\omega,\rho)}$ is slide-positive, this is not true in general.
See \cite[Example~3.12]{Assaf-Bergeron} for a revealing example.
We are especially interested in the fact that the rightmost linear order in  Figure~\ref{fig:ex_computation} contributes a \emph{single} slide polynomial.
To explore this aspect further, we need to introduce more notions attached to $\rho$-restricted linear orders.

%%%%%%%%%%%%%%%%%%%%%%%%%%%%%%%
\subsection{Linear order with restrictions}\label{subsec:lin_orders}
%%%%%%%%%%%%%%%%%%%%%%%%%%%%%%%

Consider a triple $(\p L,\omega,\rho)$ where $\p L$ is a linear order.
Observe that in view of $\omega$, some inequalities imposed by $\rho$ might be redundant.
Figure~\ref{fig:redundant_L}, on the left, depicts a linear order $\p L$ given as $1\precdot_{\p L}2 \precdot_{\p L}3\precdot_{\p L}4 \precdot_{\p L}5$. 
The labeling $\omega$ going from the minimal element in $\p L$ to the maximal gives the permutation $23154$ in one-line notation.
The restriction $\rho$ is defined by $\rho(1)=1$, $\rho(2)=4$, $\rho(3)=5$, $\rho(4)=6$, and $\rho(5)=4$.
Since $\omega(5)=4<\omega(4)=5$, we infer that a $\rho$-restricted $(\p L,\omega)$-partition $f$ must satisfy $f(4)<f(5)$ in addition to $f(4)\leq \rho(4)=6$ and $f(5)\leq \rho(5)=4$.
Clearly, the restriction $\rho(4)=6$ can be replaced with the tighter version $\rho(4)=3$ without altering the set of $\rho$-restricted $(\p L,\omega)$-partitions.

In general, by replacing each inequality imposed by $\rho$ with the tightest one, we obtain a new restriction map $\bar{\rho}^{\p L}$ with the key property that 
$\mc{A}(\p L,\omega,\rho)=\mc{A}(\p L,\omega,\bar{\rho}^{\p L})$.
Observe that $\bar{\rho}^{\p L}$ does depend on $\omega$, but we suppress this dependence and hope that no confusion results.
We next formalize this procedure of finding $\bar{\rho}^{\p L}$. 
Given a positive integer $n$ and a permutation $\pi\in S_n$, let $\p L$ be the linear order $\pi(1)\precdot_{\p L}\cdots \precdot_{\p L} \pi(n)$ endowed with labeling $\omega$ and restriction $\rho$.
We define $\bar{\rho}^{\p L}$ recursively top-down starting from the maximum element of $\p L$.
More precisely, for $i$ from $n$ down to $1$, set 
\begin{align}\label{eqn:barrho}
\bar{\rho}^{\p L}(\pi(i))\coloneqq
\left \lbrace
\begin{array}{ll}
\rho(\pi(n)) & i=n\\
\min\{\bar{\rho}^{\p L}(\pi(i+1)),\rho(\pi(i))\} & \omega\circ\pi(i)< \omega\circ\pi(i+1)\\
\min\{\bar{\rho}^{\p L}(\pi(i+1))-1, \rho(\pi(i))\}
& \omega\circ\pi(i)> \omega\circ\pi(i+1).
\end{array}
\right.
\end{align}
Note in particular  that $\bar{\rho}^{\p L}$ depends on the descent set of the permutation $\omega\circ \pi$, that is, the permutation obtained by reading the labels from the minimal element to the maximal element.
For the linear order from Figure~\ref{fig:redundant_L} encountered earlier, $\pi$ is the identity permutation, and the permutation  $\omega\circ \pi$ is $23154$.
The reader can check that $\bar{\rho}^{\p L}$ is exactly as depicted in the linear order in the middle in Figure~\ref{fig:redundant_L}.

\begin{remark}
Assaf-Bergeron \cite[Definition 3.2]{Assaf-Bergeron} discuss the procedure of find the `tightest' restriction map more generally for posets.
Their definition is therefore a bit more involved, but in the case of linear orders one obtains the description in \eqref{eqn:barrho}.
\end{remark}

\begin{figure}[h]
  \begin{tikzpicture}[scale=0.7]    
    \node[draw, circle,minimum size=10pt,inner sep=2pt, outer sep=0pt] at (0, 0) (v1) {$2$}; 
    \node[draw, circle,minimum size=10pt,inner sep=2pt, outer sep=0pt] at (0, 1.5) (v2) {$3$}; 
    \node[draw, circle,minimum size=10pt,inner sep=2pt, outer sep=0pt] at (0,3)   (v3) {$1$}; 
    \node[draw, circle,minimum size=10pt,inner sep=2pt, outer sep=0pt] at (0, 4.5)  (v4) {$5$}; 
    \node[draw, circle,minimum size=10pt,inner sep=2pt, outer sep=0pt] at (0, 6.0)  (v5) {$4$}; 
    
    \node[] at (-1,0)    (ineq1) {\textcolor{red}{\tiny{$\leq\! 1$}}}; 
    \node[] at (-1, 1.5)   (ineq2) {\textcolor{red}{\tiny{$\leq\!4$}}}; 
    \node[] at (-1,3.0)  (ineq3) {\textcolor{red}{\tiny{$\leq\! 5$}}}; 
    \node[] at (-1, 4.5)  (ineq4) {\textcolor{red}{\tiny{$\leq\! 6$}}}; 
    \node[] at (-1, 6.0)  (ineq5) {\textcolor{red}{\tiny{$\leq\! 4$}}};  
    
    \draw [black,thick]   (v1) to (v2);
    \draw [black,thick]   (v2) to (v3);
    \draw [black,thick]   (v3) to (v4);
    \draw [black,thick]   (v4) to (v5);

    \node[draw, circle,minimum size=10pt,inner sep=2pt, outer sep=0pt] at (5, 0) (w1) {$2$}; 
    \node[draw, circle,minimum size=10pt,inner sep=2pt, outer sep=0pt] at (5, 1.5) (w2) {$3$}; 
    \node[draw, circle,minimum size=10pt,inner sep=2pt, outer sep=0pt] at (5,3)   (w3) {$1$}; 
    \node[draw, circle,minimum size=10pt,inner sep=2pt, outer sep=0pt] at (5, 4.5)  (w4) {$5$}; 
    \node[draw, circle,minimum size=10pt,inner sep=2pt, outer sep=0pt] at (5, 6.0)  (w5) {$4$}; 
    
    \node[] at (4,0)    (ineq1) {\textcolor{red}{\tiny{$\leq\! 1$}}}; 
    \node[] at (4, 1.5)   (ineq2) {\bcancel{\textcolor{red}{\tiny{$\leq\!4$}}}}; 
    \node[] at (3, 1.5)   (ineq2p) {\textcolor{red}{\tiny{$\leq\!2$}}}; 
    \node[] at (4,3.0)  (ineq3) {\bcancel{\textcolor{red}{\tiny{$\leq\! 5$}}}};
    \node[] at (3,3.0)  (ineq3p) {\textcolor{red}{\tiny{$\leq\! 3$}}}; 
    \node[] at (4, 4.5)  (ineq4) {\bcancel{\textcolor{red}{\tiny{$\leq\! 6$}}}}; 
    \node[] at (3, 4.5)  (ineq4p) {\textcolor{red}{\tiny{$\leq\! 3$}}}; 
    \node[] at (4, 6.0)  (ineq5) {\textcolor{red}{\tiny{$\leq\! 4$}}};  
    
    \draw [black,thick]   (w1) to (w2);
    \draw [black,thick]   (w2) to (w3);
    \draw [black,thick]   (w3) to (w4);
    \draw [black,thick]   (w4) to (w5);

        \draw[rounded corners,fill=blue,opacity=0.3] (9.5, -0.5) rectangle (10.5, 2) {}; 

\draw[rounded corners,fill=blue,opacity=0.3] (9.5, 2.5) rectangle (10.5, 5) {}; 

\draw[rounded corners,fill=blue,opacity=0.3] (9.5, 5.5) rectangle (10.5, 6.5) {};

    \node[draw, circle,minimum size=10pt,inner sep=2pt, outer sep=0pt] at (10, 0) (v1) {$2$}; 
    \node[draw, circle,minimum size=10pt,inner sep=2pt, outer sep=0pt] at (10, 1.5) (v2) {$3$}; 
    \node[draw, circle,minimum size=10pt,inner sep=2pt, outer sep=0pt] at (10,3)   (v3) {$1$}; 
    \node[draw, circle,minimum size=10pt,inner sep=2pt, outer sep=0pt] at (10, 4.5)  (v4) {$5$}; 
    \node[draw, circle,minimum size=10pt,inner sep=2pt, outer sep=0pt] at (10, 6.0)  (v5) {$4$}; 
    
    \node[] at (9,0)    (ineq1) {\textcolor{red}{\tiny{$\leq\! 1$}}}; 
    \node[] at (9, 1.5)   (ineq2) {\textcolor{red}{\tiny{$\leq\!2$}}}; 
    \node[] at (9,3.0)  (ineq3) {\textcolor{red}{\tiny{$\leq\! 3$}}}; 
    \node[] at (9, 4.5)  (ineq4) {\textcolor{red}{\tiny{$\leq\! 3$}}}; 
    \node[] at (9, 6.0)  (ineq5) {\textcolor{red}{\tiny{$\leq\! 4$}}};  
    
    \draw [black,thick]   (v1) to (v2);
    \draw [black,thin,dashed]   (v2) to (v3);
    \draw [black,thick]   (v3) to (v4);
    \draw [black,thin,dashed]   (v4) to (v5);

  \end{tikzpicture}
  \caption{A linear order $\p L$ with redundant $\rho$ and then with $\bar{\rho}^{\p L}$.}
  \label{fig:redundant_L}
\end{figure}
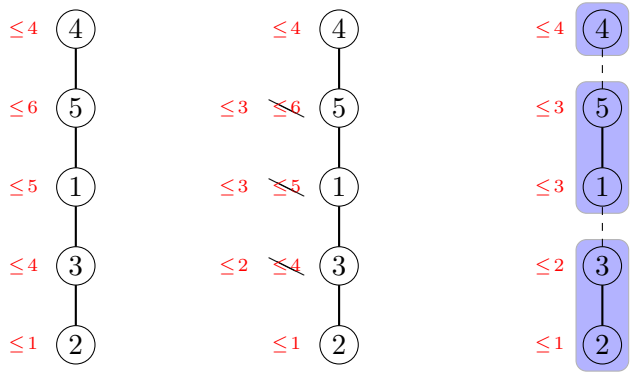

We use $\bar{\rho}^{\p L}$ to generalize the notion of descent compositions to account for the restriction map.
Define the \emph{reduced weak descent composition} of $(\p L,\omega, \rho)$, denoted by $\mathrm{rdes}(\p L,\omega, \rho)$,  as follows. 
Let $i_1<\cdots< i_k$ be all the descents in $\omega\circ\pi$.
Consider the chains $C_1,\dots,C_{k+1}$ defined by setting 
\begin{align}\label{eqn:chains}
C_j\coloneqq \pi(i_{j-1}+1)\precdot_{\p L} \cdots \precdot_{\p L} \pi(i_j),
\end{align}
where $i_0\coloneqq 0$ and $i_{k+1}\coloneqq n$.
For $j = 1,\dots, k+1$, set $c_j = \bar{\rho}^{\p L}(\pi(i_{j-1}+1))$, and define the $c_j$-th part of $\mathrm{rdes}(\p L,\omega, \rho)$  to equal $|C_j|$, and set all other parts to 0.
Note that $\pi(i_{j-1}+1)$ is the minimal element of the chain $C_j$, so the values $c_j$ are obtained by simply evaluating $\bar{\rho}^{\p L}$ at the minimal element of each chain $C_j$ from $j=1$ through $k+1$.
For the rightmost linear order in Figure~\ref{fig:redundant_L}, the dashed edges denote the descent edges whose removal results in the chains $C_1$, $C_2$ and $C_3$ from bottom to top. 
Picking the smallest value of $\bar{\rho}^{\p L}$ in each shaded region tells us that $c_1=1$, $c_2=3$, $c_3=4$.
Since $|C_1|=|C_2|=2$ and  $|C_3|=1$, we infer that 
$\mathrm{rdes}(\p L, \rho)=(2,0,2,1)$.
The reader may further verify that 
\begin{align}\label{eqn:ex_continued}
\p {F}_{(\p L,\omega,\rho)}=\slide{2021}(\x_4)+\slide{1121}(\x_4).
\end{align}
We remark here that we could have replaced the alphabet $\x_4$ above with any $\x_r$ for any $r\geq 4$.

\begin{remark}\label{rem:why_rdes}
Recall the folklore bijective correspondence between strong compositions of a nonnegative integer $n$ and subsets of $[n-1]$. Suppose that $S\subseteq [n-1]$ maps to the strong composition $\mathrm{comp}(S)$.
If $\mathrm{Des}(\pi)$ denotes the descent set of $\pi$, then from the definition of $\mathrm{rdes}(\p L,\omega,\rho)$, it follows that $\flatten{\mathrm{rdes}(\p L,\omega,\rho)}=\mathrm{comp}(\mathrm{Des}(\pi))$ \cite[Equation 3.4]{Assaf-Bergeron}, which explains the name reduced weak descent composition.
\end{remark}

In our example, we see that $\slide{\mathrm{rdes}(\p L,\omega, \rho)}$ is a term in the expansion of $\p {F}_{(\p L,\omega,\rho)}$ in slide polynomials.
We are particularly interested in  case where it is the \emph{only} term. 
To this end, we have the following result of Assaf-Bergeron \cite[Proposition 3.10]{Assaf-Bergeron}.

\begin{proposition}\label{prop:single_slide}
Consider a $\rho$-restricted labeled poset $(\p L,\omega,\rho)$ where $\p L$ is a linear order on $[n]$ given by $\pi(1)\precdot_{\p L}\cdots \precdot_{\p L} \pi(n)$.
Suppose that for all $i\in[n-1]$, we have that  $\omega\circ\pi(i) < \omega\circ\pi(i+1)$ implies
$\bar{\rho}^{\p L}(\pi(i)) = \bar{\rho}^{\p L}(\pi(i+1))$.
Then we have that 
\[
	 \p F_{(\p L,\omega,\rho)} = \slide{\mathrm{rdes}(\p L,\omega, \rho)}(\x_r).
\]
where $r\geq \max\{\bar{\rho}^{\p L}(\pi(i)\}_{i\in [n]}$.
\end{proposition}
Going back to the rightmost linear order in Figure~\ref{fig:ex_computation}, we see that $\p F_{(\p L,\omega,\rho)}$ equals the slide polynomial $\slide{021}(\x_3)$ by Proposition~\ref{prop:single_slide}.

%%%%%%%%%%%%%%%%%%%%%%%%%%%%%%%
\section{Slide-positivity of $\p X_D(\x_r;t)$}\label{sec:main_result}
%%%%%%%%%%%%%%%%%%%%%%%%%%%%%%%

We proceed to establish our central result that $\p X_D(\x_r;t)$ expands in terms of  slide polynomials with coefficients in $\mathbb{N}[t]$.
To this end, we need more terminology.

Consider any graph $G=([n],E)$. 
For $\pi\in S_n$, define $\mathrm{inv}_G(\pi)$ to be the number of \emph{$G$-inversions} of $\pi$, that is,
\begin{align}\label{eqn:def_G_inversion}
\mathrm{inv}_G(\pi)=|\{\{i,j\}\in E\suchthat i<j \text{ and } \pi(i)>\pi(j)\}|.
\end{align} 
Given a poset $P$ on $[n]$, we say that $i \in [n-1]$ is a \emph{$P$-descent}
 of $\pi$ if $\pi(i+1) \prec_P \pi(i) $. 
 We denote the set of $P$-descents of $\pi$ by $\mathrm{Des}_P (\pi)$. 
 If $i$ and $j$ are comparable in $P$, we denote this by $i\sim_P j$. 
 Otherwise, we write $i\nsim_P j$.
Recall that  the \emph{incomparability graph} of a poset $P$ is the simple graph whose vertex set is $[n]$ and  edges are given by $\{i,j\}$ where $i\nsim_P j$.

\textbf{From this point onward, fix a partial Dyck path $D\in P_{n,r}$}. 
Let $G\coloneqq G_D$ be the corresponding Dyck graph, and let $\rho\coloneqq \rho_D$ be the restriction induced by $D$. Let $E$ denote the set of edges of $G$.
 We realize $G$ as the incomparability graph of a poset $P_D$ on $[n]$ as follows: declare $i\prec_{P_D} j$ if and only if $j>i$ and $\{i,j\}\notin E$.
Given $\pi \in S_n$, let $\p L_{\pi}$ be the linear order on the vertices of $G$ given by $\pi(1)\precdot_{\p L_{\pi}}\cdots \precdot_{\p L_{\pi}}\pi(n)$.
This linear order induces an acyclic orientation $\bfo\coloneqq \bfo_{\pi}$ of $G$ obtained by directing  $\{\pi(i),\pi(j)\}\in E$ for $i<j$ from $\pi(j)$ to $\pi(i)$.
Observe that  $\bfo$  gives rise to a poset $P_{\bfo}$ on $[n]$ by taking the transitive closure of the relation given by $j\prec_{P_{\bfo}} i$ if there is a directed edge from $i$ to $j$ in $\bfo$.
Furthermore, $P_{\bfo}$ and ${\p L_{\pi}}$ also inherit the restriction map $\rho$.
For the acyclic orientation $\bfo$ in Figure~\ref{fig:ao_gd}, the $P_{\bfo}$ is shown on the left in Figure~\ref{fig:p_o}.
A permutation $\pi$ which induces this $\bfo$ is $645123$.
 
\begin{figure}[h]
\includegraphics{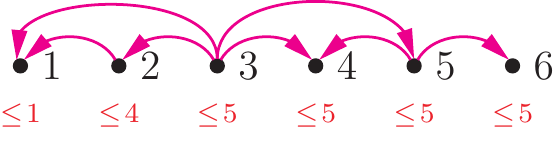}
\caption{An acyclic orientation $\bfo$ of the graph in Figure~\ref{fig:g_d}.}
\label{fig:ao_gd}
\end{figure}

We say that a coloring $f$ of $G$ is \emph{compatible} with $\bfo$ if for every directed edge $i\to j$ we have $f(i)>f(j)$.
For  $\p X_{D}(\x_r;t)$, we are interested in the proper colorings $f:[n]\to \mathbb{Z}_{> 0}$ that further satisfy $f(i)\leq \rho(i)$.
Since every proper coloring is compatible with a unique acyclic orientation, we can partition the set of $\rho$-restricted proper colorings based on compatibility.
%Thus, it suffices to the generating function of colorings that respect $\bfo$ for a fixed $\bfo$ and meet the constraints imposed by $\rho_D$.
This leads us to interpret such proper colorings as $\rho$-restricted $(P_{\bfo},\omega)$-partitions for an appropriate $\omega$.

We seek  a labeling $\ombfo$ that induces strict inequalities on all cover relations in $P_{\bfo}$. 
More precisely, since we are interested in colorings compatible with $\bfo$, our labeling $\ombfo$ must satisfy  $\ombfo(i)<\ombfo(j)$ if $j\precdot_{P_{\bfo}} i$.
We construct $\ombfo$ as follows.
Initialize $\mathsf{ctr}$ to $ 1$, $\p G$ to $G$, and perform the following steps. 
\begin{enumerate}
\item Find the largest $i$ such that vertex $i$ in $\p G$ has indegree $0$ with respect to $\mathbf{o}$.
\item Set $\ombfo(i)=\mathsf{ctr}$, and subsequently increment $\mathsf{ctr}$ by $1$.
\item Remove the vertex $i$ along with all edges incident to it, and let $\p G$ be the new directed graph obtained with the acyclic orientation inherited from $\mathbf{o}$. 
If there is at least one vertex in $\p G$, return to step $(1)$, else terminate.
\end{enumerate}
For the $\bfo$ in Figure~\ref{fig:ao_gd}, this algorithm gives $\omega(1)=6$, $\omega(2)=5$, $\omega(3)=1$, $\omega(4)=4$, $\omega(5)=2$, and $\omega(6)=3$.
The triple $(P_{\bfo},\ombfo,\rho)$ is depicted on the right in Figure~\ref{fig:p_o}.
On the left, the numbers outside nodes represent the numbering inherited from the graph. On the right, the numbers within nodes represent the labeling $\ombfo$.
\begin{figure}[h]
  \begin{tikzpicture}[scale=0.7]
    \node[draw, circle,minimum size=4pt,inner sep=2pt, outer sep=0pt,fill=black] at (-4, 0) [label=left:1]   (v1) {}; 
    \node[draw, circle,minimum size=4pt,inner sep=2pt, outer sep=0pt,fill=black] at (-4, 2) [label=right:2]  (v2) {}; 
    \node[draw, circle,minimum size=4pt,inner sep=2pt, outer sep=0pt,fill=black] at (-2,4) [label=right:3] (v3) {}; 
    \node[draw, circle,minimum size=4pt,inner sep=2pt, outer sep=0pt,fill=black] at (-1, 0) [label=left:4]  (v4) {}; 
    \node[draw, circle,minimum size=4pt,inner sep=2pt, outer sep=0pt,fill=black] at (0, 2) [label=right:5]   (v5) {}; 
    \node[draw, circle,minimum size=4pt,inner sep=2pt, outer sep=0pt,fill=black] at (1, 0) [label= right:6] (v6) {}; 
    
    \node[] at (-4,-0.5)    (ineq1) {\textcolor{red}{\small{$\leq\! 1$}}}; 
    \node[] at (-4, 2.5)   (ineq2) {\textcolor{red}{\small{$\leq\!4$}}}; 
    \node[] at (-2,4.5)  (ineq3) {\textcolor{red}{\small{$\leq\! 5$}}}; 
    \node[] at (-1, -0.5)  (ineq4) {\textcolor{red}{\small{$\leq\! 5$}}}; 
    \node[] at (0, 2.5)  (ineq5) {\textcolor{red}{\small{$\leq\! 5$}}}; 
    \node[] at (1, -0.5)   (ineq6) {\textcolor{red}{\small{$\leq\! 5$}}}; 
    
    \draw [black,thick]   (v1) to (v2);
    %\draw [black,thick]   (v1) to (v3);
    \draw [black,thick]   (v2) to (v3);
    %\draw [black,thick]   (v3) to (v4);
    \draw [black,thick]   (v3) to(v5);
    \draw [black,thick]   (v4) to (v5);
    \draw [black,thick]   (v5) to (v6);
    
    \node[draw, circle,minimum size=10pt,inner sep=0pt, outer sep=0pt] at (5, 0) (v1) {$6$}; 
    \node[draw, circle,minimum size=10pt,inner sep=0pt, outer sep=0pt] at (5, 2) (v2) {$5$}; 
    \node[draw, circle,minimum size=10pt,inner sep=0pt, outer sep=0pt] at (7,4)   (v3) {$1$}; 
    \node[draw, circle,minimum size=10pt,inner sep=0pt, outer sep=0pt] at (8, 0)  (v4) {$4$}; 
    \node[draw, circle,minimum size=10pt,inner sep=0pt, outer sep=0pt] at (9, 2)  (v5) {$2$}; 
    \node[draw, circle,minimum size=10pt,inner sep=0pt, outer sep=0pt] at (10, 0)  (v6) {$3$};

    \node[] at (5,-0.5)    (ineq1) {\textcolor{red}{\small{$\leq\! 1$}}}; 
    \node[] at (5, 2.6)   (ineq2) {\textcolor{red}{\small{$\leq\!4$}}}; 
    \node[] at (7,4.5)  (ineq3) {\textcolor{red}{\small{$\leq\! 5$}}}; 
    \node[] at (10, -0.5)  (ineq4) {\textcolor{red}{\small{$\leq\! 5$}}}; 
    \node[] at (9, 2.60)  (ineq5) {\textcolor{red}{\small{$\leq\! 5$}}}; 
    \node[] at (8, -0.5)   (ineq6) {\textcolor{red}{\small{$\leq\! 5$}}}; 
    
    \draw [black,thick]   (v1) to (v2);
    % \draw [black,thick]   (v1) to (v3);
    \draw [black,thick]   (v2) to (v3);
    %\draw [black,thick]   (v3) to (v4);
    \draw [black,thick]   (v3) to(v5);
    \draw [black,thick]   (v4) to (v5);
    \draw [black,thick]   (v5) to (v6); 
    
  \end{tikzpicture}
  \caption{The poset $P_{\bfo}$ (left) and the triple $(P_{\bfo},\ombfo,\rho)$ (right).}
  \label{fig:p_o}
\end{figure}
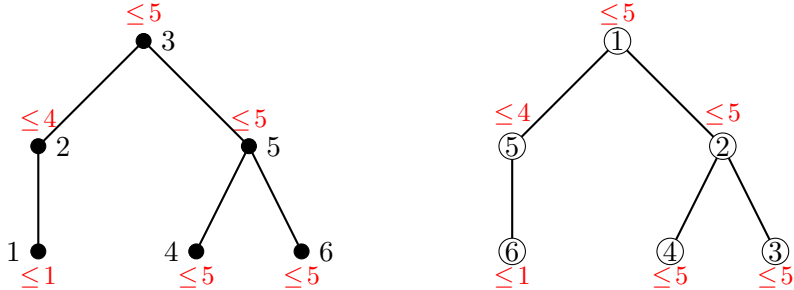

\noindent We are ready to establish a key lemma that relates $P_D$-descents of $\pi$ and ascents in $\ombfo\circ \pi$.
\begin{lemma}\label{lem:key_lemma}
For $i\in [n-1]$, we have that
\[
\pi(i)\succ_{P_D}\pi(i+1) \Longleftrightarrow \ombfo\circ \pi (i) <\ombfo\circ\pi(i+1).
\]
\end{lemma}
\begin{proof}
We establish the forward implication first.
Assume $\pi(i)\succ_{P_D} \pi(i+1)$. 
From the definition of $P_D$ we infer the following two facts: first, $\pi(i)>\pi(i+1)$, and second, $\{\pi(i+1),\pi(i)\}\notin E$. 

We claim that $\pi(i)\nsim_{P_{\bfo}}\pi(i+1)$. 
Indeed, if this is not the case, then there exists a directed path from $\pi(i+1)$ to $\pi(i)$ in $\bfo$, and thereby, a vertex $\pi(j)$ where $j\neq i,i+1$ that lies on this  path.
Since there is a directed path from $\pi(i+1)$ to $\pi(j)$, we infer that $i+1>j$. On the other hand, the directed path from $\pi(j)$ to $\pi(i)$ implies $j>i$. It follows that $i+1>j>i$, which is clearly absurd.

Next we show that $\ombfo(\pi(i))<\ombfo(\pi(i+1))$.
Assume to the contrary that $ \ombfo(\pi(i)) >\ombfo(\pi(i+1))$, and consider the instant in our labeling algorithm when $\pi(i+1)$ is assigned a label.
As $\pi(i)>\pi(i+1)$ and $\pi(i)$ is unlabeled at that instant, there is an unlabeled vertex $\pi(i_1)$ such that there is an edge in $\bfo$ directed from $\pi(i_1)$ to $\pi(i)$.
Furthermore, $\pi(i_1)\nsim_{P_{\bfo}}\pi(i+1)$. If not, the existence of a directed path from $\pi(i+1)$ to $\pi(i_1)$ would imply that $\pi(i+1)\sim_{P_{\bfo}}\pi(i)$, which is false.
On the other hand, the existence of a directed path from $\pi(i_1)$ to $\pi(i+1)$ would contradict the fact that our labeling algorithm assigns a label to $\pi(i+1)$ before $\pi(i_1)$.

Repeating this argument, we can construct a directed path $\pi(i_k) \to \pi(i_{k-1}) \to \cdots \to\pi(i_1)\to \pi(i_0)=\pi(i)$
with the property that for $0\leq j\leq k$, 
each $\pi(i_j)$ is unlabeled and satisfies $\pi(i_j)\nsim_{P_{\bfo}}\pi(i+1)$. 
Consider a maximal such path.
Then all vertices with edges directed towards $\pi(i_k)$ are already labeled.
This in turn implies that $\pi(i_k)<\pi(i+1)$, as the opposite inequality implies that $\pi(i_k)$ has a smaller label $\pi(i+1)$, which is not the case.

Now note that there must exist a $1\leq j\leq k$ such that $\pi(i_j)<\pi(i+1)$ but 
$\pi(i_{j-1}) > \pi(i+1)$.
Indeed, if such a $j$ did not exist, then the edge $\{\pi(i_1),\pi(i)\}$ would imply that 
$\{\pi(i+1),\pi(i)\}\in E$ as $G$ is a Dyck graph.
But we have already established that $\{\pi(i+1),\pi(i)\}\notin E$.
Now pick any $j$ whose existence we just established.
Again using the fact that  $G$ is a Dyck graph, we conclude that $\{\pi(i+1),\pi(i_{j-1})\}\in E$, which in turn implies that $\pi(i_{j-1})\sim_{P_{\bfo}}\pi(i+1)$, which is false. This completes the proof of the forward direction.

We keep our exposition on the reverse implication brief.
Assuming $\ombfo\circ\pi(i)<\ombfo\circ \pi (i+1)$, we need to show that $\pi(i)\succ_{P_D} \pi(i+1)$.
Once again, observe that $\pi(i)\nsim_{P_{\bfo}}\pi(i+1)$.
If not, we would infer the existence of a directed path from $\pi(i+1)$ to $\pi(i)$ in $\bfo$, which in turn would contradict our hypothesis that $\ombfo(\pi(i))<\ombfo(\pi(i+1))$.
To establish that $\pi(i)\succ_{P_D} \pi(i+1)$,  we proceed by contradiction.
There are two possibilities if $\pi(i) \nsucc_{P_D}\pi(i+1)$: either $\pi(i)\nsim_{P_D}\pi(i+1)$ or $\pi(i+1)\succ_{P_D} \pi(i)$.
In the former, we have $\{\pi(i),\pi(i+1)\}\in E$, which then is necessarily directed from $\pi(i+1)$ to $\pi(i)$ in $\bfo$.
This implies $\ombfo(\pi(i+1))<\ombfo(\pi(i))$, contrary to our assumption.
Finally, the case $\pi(i+1)\succ_{P_D} \pi(i)$ remains.
The argument for this is very similar to that presented in the proof of the forward direction. We omit the details.
\end{proof}

In view of the previous lemma, we now establish that $\p F_{(\p L_{\pi},\ombfo,\rho)}$ is equal to a slide polynomial. 
To this end, we recast the reduced weak descent composition $\mathrm{rdes}(\p L_{\pi},\ombfo,\rho)$ defined in Subsection~\ref{subsec:lin_orders}  in terms of $P_D$-descents of $\pi$ rather than the descent set of $\ombfo\circ \pi$.
Indeed, using Lemma~\ref{lem:key_lemma}, we can redefine $\bar{\rho}^{\p L_{\pi}}$ as follows:
For $i$ from $n$ down to $1$, set
\begin{align}\label{eqn:barrho_part2}
\bar{\rho}^{\p L_{\pi}}(\pi(i))\coloneqq
\left \lbrace
\begin{array}{ll}
\rho(\pi(n)) & i=n\\
\min\{\bar{\rho}^{\p L_{\pi}}(\pi(i+1)),\rho(\pi(i))\} & \pi(i)\succ_{P_D}\pi(i+1) \\
\min\{\bar{\rho}^{\p L_{\pi}}(\pi(i+1))-1, \rho(\pi(i))\}
& \pi(i)\nsucc_{P_D}\pi(i+1).
\end{array}
\right.
\end{align}
The above formulation removes the dependence of the recursive definition of $\bar{\rho}^{\p L_{\pi}}$ on $\ombfo$.
The procedure for computing $\mathrm{rdes}(\p L_{\pi},\ombfo,\rho)$ is the same, except the role played by descents of $\ombfo\circ\pi$ is now essayed by $P_D$-ascents of $\pi$.
To emphasize the suppression of $\ombfo$, we write $\mathrm{rdes}(\p L_{\pi},\rho)$ instead of $\mathrm{rdes}(\p L_{\pi},\ombfo,\rho)$.
The following result explains how slide polynomials enter our picture.
\begin{lemma}\label{lem:equals_single_slide}
The weight generating function $\p F_{(\p L_{\pi},\ombfo,\rho)}$ equals the slide polynomial $\slide{\mathrm{rdes}(\p L_{\pi},\rho)}(\x_r)$.
\end{lemma}
\begin{proof}
%Set $\p L\coloneqq \p L_{\pi}$.
By Proposition~\ref{prop:single_slide} and Lemma~\ref{lem:key_lemma}, it suffices to verify that $\pi(i)\succ_{P_D} \pi(i+1)$ implies $\bar{\rho}^{\p L_{\pi}}(\pi(i))=\bar{\rho}^{\p L_{\pi}}(\pi(i+1))$.
This is immediate as $\pi(i)\succ_{P_D}\pi(i+1)$ implies $\rho(\pi(i))\geq \rho(\pi(i+1))$. 
The recursive description of $\bar{\rho}^{\p L_{\pi}}$ in \eqref{eqn:barrho_part2} implies the claim. 
Note that our choice of the alphabet $\x_r$ is justified as $\rho(i)\leq r$ for all $i\in [n]$.
\end{proof}

As an example, consider the graph $G$ coming from a partial Dyck path $D\in P_{3,3}$ in Figure~\ref{fig:partial_dyck_2}.
The corresponding $P_D$ has only relation: $3\succ_{P_D} 1$.
The six linear orders on the vertices of $G$ along with the modified restrictions $\bar{\rho}^{\p L_{\pi}}$ are depicted in Figure~\ref{fig:six_linear_orders}. 
The $\p F_{(\p L_{\pi},\ombfo,\rho)}$ corresponding to each linear order from left to right are: $\slide{(1,1,1)}(\x_3)$, $\slide{(1,1,1)}(\x_3)$, $\slide{(2,0,1)}(\x_3)$, $\slide{(1\vbar 2)}(\x_3)$, $\slide{(1\vbar 1,0,1)}(\x_3)$, and $\slide{(1,1\vbar 1)}(\x_3)$.
Note that although the last three of these are in fact $0$, they  will acquire meaning in Subsection~\ref{subsec:backstable}.

\begin{figure}[h]
\begin{tikzpicture}[scale=.4]
    \draw[gray,very thin] (0,0) grid (8,8);
    \draw[line width=0.25mm, black, <->] (0,1)--(8,1);
    \draw[line width=0.25mm, black, <->] (1,0)--(1,8);
    
    \draw[rounded corners,fill=blue,opacity=0.3] (1, 1) rectangle (2, 2) {}; 
    \draw[rounded corners,fill=blue,opacity=0.3] (2, 2) rectangle (3, 3) {}; 
    \draw[rounded corners,fill=blue,opacity=0.3] (3, 3) rectangle (4, 4) {}; 
      \draw[rounded corners,fill=red,opacity=0.3] (4, 4) rectangle (5, 5) {}; 
    \draw[rounded corners,fill=red,opacity=0.3] (5, 5) rectangle (6, 6) {}; 
    \draw[rounded corners,fill=red,opacity=0.3] (6, 6) rectangle (7, 7) {};  
         
    \node[draw, circle,minimum size=3pt,inner sep=0pt, outer sep=0pt, fill=blue] at (1, 4)   (start) {};
    \node[draw, circle,minimum size=3pt,inner sep=0pt, outer sep=0pt, fill=blue] at (7, 7)   (end) {}; 
    \node[] at (6.5, 6.5)   (v1) {$6$}; 

  \node[] at (1.5, 1.5)   (c1) {$1$}; 
  \node[] at (2.5, 2.5)   (c2) {$2$}; 
  \node[] at (3.5, 3.5)   (c3) {$3$}; 
  \node[] at (4.5, 4.5)   (c4) {$4$}; 
  \node[] at (5.5, 5.5)   (c5) {$5$}; 
  
    \draw[blue, line width=0.3mm] (1,4)--(2,4);
    \draw[blue, line width=0.3mm] (2,4)--(2,5);
    \draw[blue, line width=0.3mm] (2,5)--(4,5);
    \draw[blue, line width=0.3mm] (4,5)--(4,6);
    \draw[blue, line width=0.3mm] (4,6)--(5,6);
    \draw[blue, line width=0.3mm] (5,6)--(5,7);
    \draw[blue, line width=0.3mm] (5,7)--(7,7);
    
    \node[draw, circle,minimum size=4pt,inner sep=0pt, outer sep=0pt, fill=black] at (14, 3) [label=right:1]   (v1) {}; 
    \node[draw, circle,minimum size=4pt,inner sep=0pt, outer sep=0pt, fill=black] at (17, 3) [label=right:2]  (v2) {}; 
    \node[draw, circle,minimum size=4pt,inner sep=0pt, outer sep=0pt, fill=black] at (20,3)  [label=right:3] (v3) {}; 
    
    \draw [black,thick]   (v1) to[out=60,in=120] (v2);
    \draw [black,thick]   (v2) to[out=60,in=120] (v3);
    
     \node[] at (14,2.25)    (ineq1) {\textcolor{red}{\tiny{$\leq\! 1$}}}; 
    \node[] at (17, 2.25)   (ineq2) {\textcolor{red}{\tiny{$\leq\!3$}}}; 
    \node[] at (20,2.25)  (ineq3) {\textcolor{red}{\tiny{$\leq\! 3$}}}; 
   
  \end{tikzpicture}
\captionof{figure}{A partial Dyck path $D\in P_{3,3}$ and associated $G_D$.}
\label{fig:partial_dyck_2}
\end{figure}
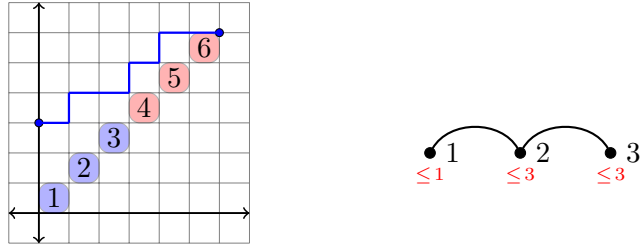

\begin{figure}[h]
\centering
  \begin{tikzpicture}[scale=0.8]    
         \node[draw, circle,minimum size=4pt,inner sep=2pt, outer sep=0pt,fill=black] at (-8, 0) [label=left:1] (a1) {}; 
    \node[draw, circle,minimum size=4pt,inner sep=2pt, outer sep=0pt,fill=black] at (-8, 1.5) [label=left:2] (a2) {}; 
    \node[draw, circle,minimum size=4pt,inner sep=2pt, outer sep=0pt,fill=black] at (-8,3) [label=left:3]   (a3) {};   

 \node[draw, circle,minimum size=4pt,inner sep=2pt, outer sep=0pt,fill=black] at (-5, 0) [label=left:1] (b1) {}; 
    \node[draw, circle,minimum size=4pt,inner sep=2pt, outer sep=0pt,fill=black] at (-5, 1.5)[label=left:3]  (b2) {}; 
    \node[draw, circle,minimum size=4pt,inner sep=2pt, outer sep=0pt,fill=black] at (-5,3) [label=left:2]   (b3) {};

     \node[draw, circle,minimum size=4pt,inner sep=2pt, outer sep=0pt,fill=black] at (-2, 0) [label=left:3] (c1) {}; 
    \node[draw, circle,minimum size=4pt,inner sep=2pt, outer sep=0pt,fill=black] at (-2, 1.5)[label=left:1]  (c2) {}; 
    \node[draw, circle,minimum size=4pt,inner sep=2pt, outer sep=0pt,fill=black] at (-2,3) [label=left:2]   (c3) {};   

 \node[draw, circle,minimum size=4pt,inner sep=2pt, outer sep=0pt,fill=black] at (1, 0) [label=left:2]  (d1) {}; 
    \node[draw, circle,minimum size=4pt,inner sep=2pt, outer sep=0pt,fill=black] at (1, 1.5) [label=left:3]  (d2) {}; 
    \node[draw, circle,minimum size=4pt,inner sep=2pt, outer sep=0pt,fill=black] at (1,3) [label=left:1]   (d3) {};   
   
     \node[draw, circle,minimum size=4pt,inner sep=2pt, outer sep=0pt,fill=black] at (4, 0) [label=left:2]  (e1) {}; 
    \node[draw, circle,minimum size=4pt,inner sep=2pt, outer sep=0pt,fill=black] at (4, 1.5) [label=left:1]  (e2) {}; 
    \node[draw, circle,minimum size=4pt,inner sep=2pt, outer sep=0pt,fill=black] at (4,3) [label=left:3]   (e3) {};   

 \node[draw, circle,minimum size=4pt,inner sep=2pt, outer sep=0pt,fill=black] at (7, 0) [label=left:3] (f1) {}; 
    \node[draw, circle,minimum size=4pt,inner sep=2pt, outer sep=0pt,fill=black] at (7, 1.5) [label=left:2]  (f2) {}; 
    \node[draw, circle,minimum size=4pt,inner sep=2pt, outer sep=0pt,fill=black] at (7,3) [label=left:1]   (f3) {};

    \node[] at (-7.50,0)    (ineq1) {\textcolor{red}{\tiny{$\leq\! 1$}}}; 
    \node[] at (-7.50, 1.5)   (ineq2) {\textcolor{red}{\tiny{$\leq\!2$}}}; 
    \node[] at (-7.50,3)  (ineq3) {\textcolor{red}{\tiny{$\leq\! 3$}}};   

\node[] at (-4.50,0)    (ineq4) {\textcolor{red}{\tiny{$\leq\! 1$}}}; 
    \node[] at (-4.50, 1.5)   (ineq5) {\textcolor{red}{\tiny{$\leq\!2$}}}; 
    \node[] at (-4.5,3)  (ineq6) {\textcolor{red}{\tiny{$\leq\! 3$}}};

\node[] at (-1.50,0)    (ineq4) {\textcolor{red}{\tiny{$\leq\! 1$}}}; 
    \node[] at (-1.50, 1.5)   (ineq5) {\textcolor{red}{\tiny{$\leq\!1$}}}; 
    \node[] at (-1.5,3)  (ineq6) {\textcolor{red}{\tiny{$\leq\! 3$}}}; 

\node[] at (1.50,0)    (ineq1) {\textcolor{red}{\tiny{$\leq\! 0$}}}; 
    \node[] at (1.50, 1.5)   (ineq2) {\textcolor{red}{\tiny{$\leq\!1$}}}; 
    \node[] at (1.50,3)  (ineq3) {\textcolor{red}{\tiny{$\leq\! 1$}}};   

\node[] at (4.50,0)    (ineq4) {\textcolor{red}{\tiny{$\leq\! 0$}}}; 
    \node[] at (4.50, 1.5)   (ineq5) {\textcolor{red}{\tiny{$\leq\!1$}}}; 
    \node[] at (4.5,3)  (ineq6) {\textcolor{red}{\tiny{$\leq\! 3$}}};

\node[] at (7.70,0)    (ineq4) {\textcolor{red}{\tiny{$\leq\! -1$}}}; 
    \node[] at (7.50, 1.5)   (ineq5) {\textcolor{red}{\tiny{$\leq\!0$}}}; 
    \node[] at (7.5,3)  (ineq6) {\textcolor{red}{\tiny{$\leq\! 1$}}};

    \draw [black,thick]   (a1) to (a2);
    \draw [black,thick]   (a2) to (a3);
    \draw [black,thick]   (b1) to (b2);
    \draw [black,thick]   (b2) to (b3);
    \draw [black,thick]   (c1) to (c2);
    \draw [black,thick]   (c2) to (c3);
    \draw [black,thick]   (d1) to (d2);
    \draw [black,thick]   (d2) to (d3);
    \draw [black,thick]   (e1) to (e2);
    \draw [black,thick]   (e2) to (e3);
    \draw [black,thick]   (f1) to (f2);
    \draw [black,thick]   (f2) to (f3);

\draw[rounded corners,fill=blue,opacity=0.3] (-8.25, -0.25) rectangle (-7.75,0.25) {}; 
\draw[rounded corners,fill=blue,opacity=0.3] (-8.25, 1.25) rectangle (-7.75,1.75) {}; 
\draw[rounded corners,fill=blue,opacity=0.3] (-8.25, 2.75) rectangle (-7.75,3.25) {}; 

\draw[rounded corners,fill=blue,opacity=0.3] (-5.25, -0.25) rectangle (-4.75,0.25) {}; 
\draw[rounded corners,fill=blue,opacity=0.3] (-5.25, 1.25) rectangle (-4.75,1.75) {}; 
\draw[rounded corners,fill=blue,opacity=0.3] (-5.25, 2.75) rectangle (-4.75,3.25) {}; 

\draw[rounded corners,fill=blue,opacity=0.3] (-2.25, -0.25) rectangle (-1.75,1.75) {}; 
%\draw[rounded corners,fill=blue,opacity=0.3] (-2.25, 1.25) rectangle (-1.75,1.75) {}; 
\draw[rounded corners,fill=blue,opacity=0.3] (-2.25, 2.75) rectangle (-1.75,3.25) {}; 

\draw[rounded corners,fill=blue,opacity=0.3] (0.75, -0.25) rectangle (1.25,0.25) {}; 
%\draw[rounded corners,fill=blue,opacity=0.3] (0.75, 1.25) rectangle (1.25,1.75) {}; 
\draw[rounded corners,fill=blue,opacity=0.3] (0.75, 1.25) rectangle (1.25,3.25) {}; 

\draw[rounded corners,fill=blue,opacity=0.3] (3.75, -0.25) rectangle (4.25,0.25) {}; 
\draw[rounded corners,fill=blue,opacity=0.3] (3.75, 1.25) rectangle (4.25,1.75) {}; 
\draw[rounded corners,fill=blue,opacity=0.3] (3.75, 2.75) rectangle (4.25,3.25) {}; 

\draw[rounded corners,fill=blue,opacity=0.3] (6.75, -0.25) rectangle (7.25,0.25) {}; 
\draw[rounded corners,fill=blue,opacity=0.3] (6.75, 1.25) rectangle (7.25,1.75) {}; 
\draw[rounded corners,fill=blue,opacity=0.3] (6.75, 2.75) rectangle (7.25,3.25) {}; 

\end{tikzpicture}
\caption{The  linear orders corresponding to $G_D$ in Figure~\ref{fig:partial_dyck_2} and the restrictions $\bar{\rho}^{\p L_{\pi}}$.}
\label{fig:six_linear_orders}
\end{figure}
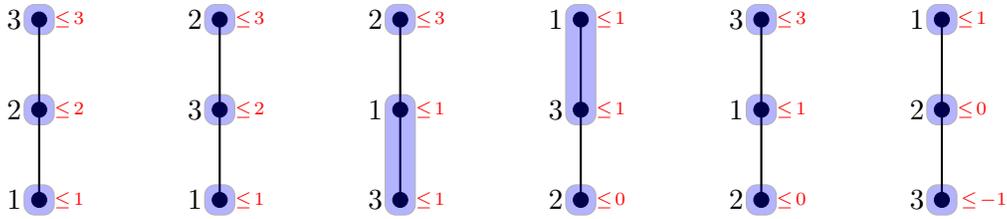

\begin{theorem}\label{thm:main_theorem}
The polynomial $\p X_D(\x;t)$ is slide-positive and we have 
\[
\p X_D(\x_r;t)=\sum_{\pi\in S_n}t^{\mathrm{inv}_G(\pi)}\slide{\mathrm{rdes}(\p L_{\pi},\rho)}(\x_r).
\]
\end{theorem}
\begin{proof}
Throughout this proof, denote by $\mathrm{PC}(G,\rho)$ the set of proper colorings $f:[n]\to \mathbb{Z}_{>0}$ of $G$ that satisfy $f(i)\leq \rho(i)$ for all $i\in [n]$.
Let $ \mathrm{AO}(G)$ be the set of acyclic orientations of $G$.
We have that
\begin{align}\label{eqn:begin}
\p X_{D}(\x_r;t)&=\sum_{f\in \mathrm{PC}(G,\rho)}t^{\des_G(f)}x_{f(1)}\cdots x_{f(n)}\nonumber\\
&=\sum_{\bfo \in \mathrm{AO}(G)}t^{|\{i\to j \text{ in } \bfo \suchthat i<j\}|}\sum_{\substack{f\in \mathrm{PC}(G,\rho)\\f \text{ compatible with } \bfo}}x_{f(1)}\dots x_{f(n)}\nonumber \\
&=\sum_{\bfo \in \mathrm{AO}(G)}t^{|\{j \prec_{P_{\bfo}}i\suchthat i<j, \{i,j\}\in E\}|}\p F_{(P_{\bfo},\ombfo,\rho)}.
\end{align}
Let $\mc{L}(P_{\bfo},\ombfo,\rho)$ denote the set of linear extensions of $P_{\bfo}$ along with the inherited labeling $\ombfo$ and restriction $\rho$.
By \cite[Corollary 3.15]{Assaf-Bergeron}, we have that $\p F_{(P_{\bfo},\ombfo,\rho)}$ may be written as a sum over elements of $\mc{L}(P_{\bfo},\ombfo,\rho)$:
\begin{align}\label{eqn:sum_over_acyclic}
\p X_D(\x_r;t)=\sum_{\bfo \in \mathrm{AO}(G)}t^{|\{j \prec_{P_{\bfo}}i\suchthat i<j, \{i,j\}\in E\}|}\sum_{\p L \in \mc{L}(P_{\bfo},\ombfo,\rho
)} \p F_{(\p L,\ombfo,\rho)}.
\end{align}

As described before, any linear order $\p L$ on the vertices of $G$, say $\pi(1)\precdot_{\p L}\cdots \precdot_{\p L}\pi(n)$, induces a unique acyclic orientation $\bfo$, in addition to uniquely determining  $\ombfo$.
This allows us to rewrite the sum on the right hand side of ~\eqref{eqn:sum_over_acyclic} as ranging over  linear orders on the vertices of $G$, or equivalently, permutations on $[n]$.
In this context, it is easily checked that
$t^{|\{j \prec_{P_{\bfo}}i\suchthat i<j, \{i,j\}\in E\}|}$ is $t^{\mathrm{inv}_G(\pi)}$. % if the linear order $\pi$ induces the acyclic orientation $\bfo$.
Lemma~\ref{lem:equals_single_slide} implies that we can replace $\p F_{(\p L,\ombfo,\rho)}$ in \eqref{eqn:sum_over_acyclic} with $\slide{\mathrm{rdes}(\p L_{\pi},\rho)}(\x_r)$.
The claim now follows.
\end{proof}
For the partial Dyck path $D\in P_{3,3}$ in Figure~\ref{fig:partial_dyck_2}, we have the following expansion:
\begin{align}
	\p X_D(\x_3;t)=\slide{111}(\x_3)+t\slide{111}(\x_3)+t\slide{201}(\x_3)+t\slide{1\vbar 2}(\x_3)+t\slide{1\vbar 101}(\x_3)+t^2\slide{11\vbar 1}(\x_3).
\end{align}
In writing our weak compositions we have omitted commas and parentheses. Recall also that the  vertical bar separates the positively indexed parts from the rest.
We proceed to address the question of how the expansion in terms of slides in the context of Dyck graphs relates to the expansion for chromatic quasisymmetric functions in terms of fundamental quasisymmetric functions.
The contents of the next subsection are heavily inspired by the recent work of Lam-Lee-Shimozono \cite{Lam-Lee-Shimozono}.

%%%%%%%%%%%%%%%%%%%%%%%%%
\subsection{The stable limit}\label{subsec:backstable}
%%%%%%%%%%%%%%%%%%%%%%%%%

For this subsection, let $\x$ denote the set of commuting indeterminates  $\{x_i\suchthat i\in \mathbb{Z}\}$ endowed with the total order $x_j<x_{j+1}$ for all $j\in \mathbb{Z}$. 
Furthermore, set $\x_{\leq r}\coloneqq \{x_i \suchthat i\leq r\}$. 
In particular, $\x_{-}\coloneqq \x_{\leq 0}$.
For ${\bf a}\in S_{\mathbb{Z}}$ we obtain a well-defined monomial 
\[
\x^{\bf a}=\prod_{i\in \mathbb{Z}}x_i^{a_i}.
\]
Let $R$ be the $\mathbb{Q}$-algebra of formal power series $f$
in the variables $x_i$
for $i\in \mathbb{Z}$ such that  $f$ has bounded total degree and there is an
$N\in \mathbb{Z}$ such that the variables $x_i$ do not appear in $f$ for $i > N$.
We say that $f$ is \emph{back-quasisymmetric} if there exists a $b\in \mathbb{Z}$ such that for any two sequences $i_1<\cdots<i_k\leq b$ and $j_1<\cdots<j_k\leq b$ and any strong composition $\alpha=(\alpha_1,\dots,\alpha_k)$, we have that the coefficient of 
$x_{i_1}^{\alpha_1}\dots x_{i_k}^{\alpha_k}$ in $f$ equals that  of 
$x_{j_1}^{\alpha_1}\dots x_{j_k}^{\alpha_k}$.
Let $\overleftarrow{R}^{\mathrm{Qsym}}$ denote the subset of back-quasisymmetric elements of $R$.

For ${\bf a}\in S_{\mathbb{Z}}$, consider the  element $\backslide{{\bf a}}(\x)\in R$ defined as
\begin{align}\label{eqn:backstable_slide}
\backslide{{\bf a}}(\x)=\sum_{{\bf b}\in\p C_{\leq {\bf a}}}\x^{\bf b}.
\end{align}
The reader should compare the definition of $\backslide{{\bf a}}(\x)$ with that of slide polynomials in \eqref{eqn:slides}. 
Replacing the indexing set $\p C_{\leq {\bf a}}^{(r)}$ by $\p C_{\leq {\bf a}}$ allows us to obtain a formal power series rather than a polynomial.
It is clear that $\backslide{{\bf a}}(\x)\in\overleftarrow{R}^{\mathrm{Qsym}}$, and we refer to it as the `backstable' limit of the slide polynomial $\backslide{{\bf a}}(\x_r)$ where $r$ is any integer such that ${\bf a}\in S_{\mathbb{Z}}^r$. %We can unambiguously replace $\x_n$ by $\x$ in $\backslide{{\bf a}}(\x_n)$.

At this stage, it should be clear how to define a backstable analogue of the chromatic nonsymmetric polynomial.
Indeed, instead of partial Dyck paths starting at the coordinates $(0,r)$, we consider \emph{infinite} partial Dyck paths $D$ that begin at $(-\infty,r)$ and take only east steps till $(0,r)$.
Then define the backstable limit $\overleftarrow{\p X}_{D}(\x;t)$ as  follows (cf. equation~\eqref{eqn:chromatic_nonsymmetric_polynomial}).
\begin{align}\label{eqn:backstable_chromatic}
\overleftarrow{\p X}_{D}(\x;t)=
\sum_{\substack{\text{proper colorings }f:[n]\to \mathbb{Z}\\f(i)\leq \rho_D(i)}}t^{\des_G(f)}x_{f(1)}\cdots x_{f(n)}
\end{align}
Note that the only deviation from the definition of the chromatic nonsymmetric polynomial is that we are allowing `nonpositive colors'. More importantly though, note that the negative colors do not play a role in determining the restriction map $\rho\coloneqq \rho_D$.
%Finally, we can replace $\x_r$ by $\x$ without causing any confusion.
It follows from definition that $\overleftarrow{\p X}_{D}(\x;t)$ is back-quasisymmetric.
In fact, one can replace each slide polynomial $\slide{\mathrm{rdes}(\p L_{\pi},\rho)}$ appearing in the expansion of $\p X_D(\x_r;t)$ in Theorem~\ref{thm:main_theorem} by $\backslide{\mathrm{rdes}(\p L_{\pi},\rho)}(\x)$ and thus obtain the following theorem.

\begin{theorem}
\label{thm:backstable_chromatic_into_slides}
The formal power series $\overleftarrow{\p X}_D(\x;t)$ is backstable slide-positive and we have the expansion
\[
\overleftarrow{\p X}_D(\x;t)=\sum_{\pi\in S_n}t^{\mathrm{inv}_G(\pi)}\backslide{\mathrm{rdes}(\p L_{\pi},\rho)}(\x).
\]
\end{theorem}

We now describe how to recover the expansion of the chromatic quasisymmetric function in terms of fundamental quasisymmetric functions \cite[Theorem 3.1]{ShWa}  from Theorem~\ref{thm:backstable_chromatic_into_slides}.
Following \cite[Section 3.4]{Lam-Lee-Shimozono}, consider the map $\eta_0$ on $\overleftarrow{R}^{\mathrm{Qsym}}$ that sets $x_i=0$ for $i>0$.
Clearly, $\eta_0(\overleftarrow{\p X}_D(\x;t))=X_D(\x_{-};t)$, where we have abused notation to denote the chromatic quasisymmetric function of $G_D$ by $X_D$.
Thus to understand the expansion in terms of fundamental quasisymmetric functions, it suffices to understand $\eta_0(\backslide{\mathrm{rdes}(\p L_{\pi},\rho)}(\x))$.

To this end, we introduce some notation and establish a general result.
Let $\mathrm{QSym}[\x_{-}]$ denote the ring of quasisymmetric functions in the variables $\{x_i\suchthat i\leq 0\}$. 
Given a strong composition $\alpha$, we denote by $F_{\alpha}(\x_{-})$ the \emph{fundamental quasisymmetric function} indexed by $\alpha$ in the ordered alphabet $\x_{-}$.
Given strong compositions $\alpha=(\alpha_1,\dots,\alpha_{\ell})$ and $\beta=(\beta_1,\dots,\beta_m)$, let $\alpha\!\cdot\! \beta\coloneqq (\alpha_1,\dots,\alpha_{\ell},\beta_1,\dots,\beta_m)$ denote \emph{concatenation} and $\alpha\!\odot\! \beta \coloneqq (\alpha_1,\dots,\alpha_{\ell}+\beta_1,\dots,\beta_m)$ denote \emph{near-concatenation}.
For instance, if $\alpha=(2,1,3)$ and $\beta=(1,2)$, then $\alpha\!\cdot\!\beta=(2,1,3,1,2)$ and $\alpha\!\odot\!\beta=(2,1,4,2)$.

Returning to our interest in understanding $\eta_0(\backslide{\mathrm{rdes}(\p L_{\pi},\rho)}(\x))$, recall that $\mathrm{rdes}(\p L_{\pi},\rho)$ belongs to $S_{\mathbb{Z}}^r$. 
Furthermore, by the recursive definition of $\bar{\rho}^{\p L_{\pi}}$, we can show that $\mathrm{rdes}(\p L_{\pi},\rho)$ possesses the following crucial property: if $\mathrm{rdes}(\p L_{\pi},\rho)_i>0$ for some $i\leq 0$, then for all $i\leq j\leq 0$, we have $\mathrm{rdes}(\p L_{\pi},\rho)_j>0$.
We refer to compositions with this property as \emph{tail-strong compositions}. 
Given a tail-strong composition ${\bf a}\in S_{\mathbb{Z}}^r$, our next lemma expresses $\backslide{{\bf a}}(\x)$  explicitly as an element of $\mathrm{QSym}[\x_{-}]\otimes \mathbb{Q}[x_1,\dots,x_r]$.
The proof of this lemma is simply a matter of unraveling the definitions and hence is omitted.
\begin{lemma}\label{lem:backstable_tail_strong}
Let ${\bf a}\in S_{\mathbb{Z}}^r $ be tail-strong.
Let $\alpha=\flatten{(\dots,a_{-2},a_{-1},a_0)}$ and ${\bf a}_{+}=(a_1,\dots,a_r)$.
A decomposition of $\flatten{{\bf a}^{+}}$ as $\gamma\odot \delta$ or $\gamma\cdot \delta$ naturally gives us  two sequences ${\bf a}_{+}^{\gamma}=(a_1^{\gamma},\dots,a_{r}^{\gamma})$ and ${\bf a}_{+}^{\delta}=(a_1^{\delta},\dots,a_{r}^{\delta})$ whose component-wise sum is ${\bf a}_{+}$.
We have the following expansion:
\[
\backslide{{\bf a}}(\x)=\sum_{\gamma\odot \delta=\flatten{{\bf a}_{+}} \text{ or } \gamma\cdot \delta=\flatten{{\bf a}_{+}}} F_{\alpha\cdot\gamma}(\x_{-})\slide{{\bf a}_{+}^{\delta}}(\x_r).
\]

\end{lemma}
We illustrate the content of the preceding theorem with an example.
Let ${\bf a}\in S_{\mathbb{Z}}^4$ be $(1,2 \vbar 0,2,0,1)$.
Then ${\bf a}_{+}=(0,2,0,1)$ and $\alpha=(1,2)$.
The decomposition of $\flatten{{\bf a}_{+}}=(2,1)$ as $\gamma\odot \delta $ where $\gamma=(1)$ and $\delta=(1,1)$ yields ${\bf a}_{+}^{\gamma}=(0,1,0,0)$ and ${\bf a}_{+}^{\delta}=(0,1,0,1)$.
This decomposition contributes $F_{(1,2,1)}(\x_{-})\slide{(0,1,0,1)}(x_1,\dots,x_4)$ to $\backslide{{\bf a}}(\x)$.
By Lemma~\ref{lem:backstable_tail_strong}, the complete expansion is 
\begin{align}
\backslide{{\bf a}}(\x)=F_{12}(\x_{-})\slide{0201}(\x_4)+F_{121}(\x_{-})\slide{0101}(\x_4)+F_{122}(\x_{-})\slide{0001}(\x_4)+F_{1221}(\x_{-}).
\end{align}

From Lemma~\ref{lem:backstable_tail_strong}, we see that $\eta_0(\backslide{{\bf a}}(\x))=F_{\flatten{\bf a}}(\x_{-})$ for any tail-strong composition.
Applying $\eta_0$ to both sides of the expansion in Theorem~\ref{thm:backstable_chromatic_into_slides}, we obtain
\begin{align}\label{eqn:chromatic_into_F}
X_{D}(\x_{-};t)=\sum_{\pi\in S_n}t^{\mathrm{inv}_G(\pi)}F_{\flatten{{\mathrm{rdes}(\p L_{\pi},\rho)}}}(\x_{-})
\end{align}
Clearly, $\flatten{{\mathrm{rdes}(\p L_{\pi},\rho)}}$ only depends on $\pi$, and from its definition we can show that it equals the strong composition of $n$ corresponding to the $[n-1]\setminus\mathrm{Des}_{P_D}(\pi)$; see Remark~\ref{rem:why_rdes}.
Thus, $\flatten{{\mathrm{rdes}\left(\p L_{\pi},\rho\right)}}=\mathrm{comp}\left([n-1]\setminus\mathrm{Des}_{P_D}(\pi) \right) =\mathrm{comp}\left(\mathrm{Des}_{P_D}(\pi)\right)^t$,  where $\alpha^t$ denotes the strong composition obtained by reflecting the ribbon diagram (in French notation) representing $\alpha$ across the line $y=x$.
In summary, we obtain the following result that the reader should compare to \cite[Theorem 3.1]{ShWa}.
\begin{corollary}
The chromatic quasisymmetric function corresponding to Dyck graphs has the following expansion in the basis of fundamental quasisymmetric functions:
\[
X_{D}(\x_{-};t)=\sum_{\pi\in S_n}t^{\mathrm{inv}_G(\pi)}F_{\mathrm{comp}(\mathrm{Des}_{P_D}(\pi))^t}(\x_{-}).
\]
\end{corollary}

\section{Further remarks}\label{sec:final}
We conclude our article with some additional remarks.
\begin{enumerate}
\item 
The precise definition of $\p X_D(\x_r; t)$ is motivated by Carlsson and Mellit's proof of the Shuffle Conjecture \cite{Carlsson-Mellit}, in which the authors define and study the ``characteristic function'' of a partial Dyck path. This characteristic function contains a symmetric and a nonsymmetric component; the nonsymmetric component is equivalent to our chromatic nonsymmetric polynomial. Carlsson and Mellit briefly consider the case where the first $r$ entries receive a different labeling, leading to different restriction functions. It would be interesting to see if these polynomials still expand positively in the slide basis, or if they lead naturally into the theory of ``basements'' in nonsymmetric polynomials \cite{Alexandersson}.

\item Haglund and Wilson \cite[Section 5.6]{Haglund-Wilson} postulated  that chromatic nonsymmetric polynomials are key-positive.
Unfortunately, we found counterexamples with Dyck graphs on 6 vertices.
That being said, a vast number of instances we considered were key-positive indeed.
This raises the question: Characterize $D$ such that $\p X_{D}(\x_{r};t)$ is key-positive.
One can ask the same question with  other interesting bases for the space of polynomials that are coarser than slide polynomials.

\end{enumerate}

\section*{Acknowledgements}
We are grateful to Sami Assaf, Jim Haglund, and Jongwon Kim for helpful discussions. 
The third author was supported by the National Science Foundation of China (No. 11701424).


\begin{thebibliography}{99}

\bibitem{Alexandersson}
{\sc P.~Alexandersson},
{\em{Non-symmetric Macdonald polynomials and Demazure-Lusztig operators}},
arxiv: \url{https://arxiv.org/abs/1602.05153}.

\bibitem{Alexandersson-Panova}
{\sc P.~Alexandersson and G.~Panova},
{\em{LLT polynomials, chromatic quasisymmetric functions and graphs with cycles}},
Disc. Math. 343 (2018) 3453--3482.

\bibitem{Assaf-Bergeron}
{\sc S.~Assaf and N.~Bergeron},
{\em{Flagged $(\p P,\rho)$-partitions}},
arxiv: \url{https://arxiv.org/abs/1904.06630}.

\bibitem{Assaf-Searles}
  {\sc S.~Assaf and D.~Searles},
  {\em {Schubert polynomials, slide polynomials, Stanley symmetric functions and quasi-Yamanouchi pipe dreams}},
  Adv. Math. 306 (2017) 89--122.


\bibitem{Athanasiadis}
{\sc C.~Athanasiadis},
{\em{Power sum expansion of chromatic quasisymmetric functions}},
Electron. J. Combin. 22 (2015) 9pp.


\bibitem{Brosnan-Chow}
{\sc P.~Brosnan and T.~Chow},
{\em{Unit Interval Orders and the Dot Action on the Cohomology of Regular Semisimple Hessenberg Varieties}},
Adv. Math. 329 (2018) 955--1001.

\bibitem{Birkhoff}
{\sc G.D.~Birkhoff},
{\em{A determinant formula for the number of ways of coloring a map}},
Ann. Math. 14 (1912) 42--46.

\bibitem{Carlsson-Mellit}
{\sc E.~Carlsson and A.~Mellit},
{\em{A proof of the shuffle conjecture}},
J. Amer. Math. Soc. 31 (2018) 661--697.

 
\bibitem{Chow}
{\sc T.~Chow},
{\em{Descents, quasi-symmetric functions, Robinson-Schensted for posets, and the chromatic symmetric function}}, 
J. Algebraic Combin. 10 (1999), 227--240.

\bibitem{Haglund-Wilson}
{\sc J.~Haglund and A.~Wilson},
{\em{Macdonald polynomials and chromatic quasisymmetric functions}},
arxiv: \url{https://arxiv.org/abs/1701.05622}.


\bibitem{MGP}
{\sc M.~Guay-Paquet},
{\em{A second proof of the Shareshian-Wachs conjecture, by way of a new Hopf algebra}},
arxiv: \url{https://arxiv.org/abs/1601.05498}.


\bibitem{Lam-Lee-Shimozono}
{\sc T.~Lam, S.J.~Lee and M.~Shimozono},
{\em{Back stable Schubert calculus}},
arxiv: \url{https://arxiv.org/abs/1806.11233}.

\bibitem{Pechenik-Searles}
{\sc O.~Pechenik and D.~Searles},
{\em{Asymmetric function theory}},
arxiv: \url{https://arxiv.org/abs/1904.01358}.


\bibitem{Macdonald}
{\sc I.~G. Macdonald}, {\em{Symmetric functions and {H}all polynomials}}, Oxford Classic
Texts in the Physical Sciences, The Clarendon Press, New York  (1995).

\bibitem{Novelli-Thibon}
{\sc J.-C.~Novelli and J.-Y.~Thibon},
{\em{Noncommutative unicellular LLT polynomials}},
arxiv: \url{https://arxiv.org/abs/1907.00077}.

\bibitem{Sagan}
{\sc  B. Sagan}, 
{\em{The symmetric group. Representations, combinatorial algorithms, and symmetric functions}}, Graduate Texts in Mathematics 203,  Springer-Verlag, New York (2001).
 
 
\bibitem{ShWa}
{\sc J.~Shareshian and M.~Wachs},
{\em{Chromatic quasisymmetric functions}},
Adv. Math. 295 (2016) 497--551.

\bibitem{Stanley-chromatic}
{\sc R.~P.~Stanley},
{\em{A symmetric function generalization of the chromatic polynomial of a graph}},
Adv. Math. 111 (1995) 166--194.

\bibitem{Stanley-ec2}
{\sc R.~P.~Stanley},
{\em{Enumerative Combinatorics. {V}ol. 2}},
Cambridge Studies in Advanced Mathematics, Cambridge University Press, Cambridge (1999).


\bibitem{Whitney}
{\sc H.~Whitney},
{\em{The coloring of graphs}},
Ann. Math. 33 (1932) 688--718.

\end{thebibliography}
\end{document}